\documentclass[12pt,english,reqno]{amsart}
\usepackage[T1]{fontenc}
\usepackage{enumitem}

\usepackage{mathrsfs}
\usepackage{amstext}
\usepackage{amsthm}
\usepackage{amssymb}
\usepackage{hyperref}
\hypersetup{
     colorlinks   = true,
     citecolor    = blue,
     linkcolor    = blue
}
\usepackage{pdfsync}
\usepackage{verbatim}
\usepackage{xcolor}

\usepackage[left=1in, right=1in, top=1.1in,bottom=1.1in]{geometry}
\usepackage{stmaryrd}

\newtheorem{theorem}{Theorem}[section]
\newtheorem{cor}[theorem]{Corollary}
\newtheorem{lemma}[theorem]{Lemma}
\newtheorem{proposition}[theorem]{Proposition}
\newtheorem{definition}[theorem]{Definition}

\newtheorem{hyp}[theorem]{Hypothesis}
\newtheorem{notation}[theorem]{Notation}
\theoremstyle{remark}    
\newtheorem{remark}[theorem]{Remark}

\numberwithin{equation}{section}

\newcommand{\dom}{\mathrm{Dom}}

\newcommand{\va}{\mathrm{var}}

\def\RR{\mathbb{R}}

\def\rr{\rrbracket}
\def\ll{\llbracket}

\newcommand{\bp}{\mathbf{P}}

\newcommand{\bx}{{\bf x}}
\newcommand{\bxx}{{\bf X}}
\newcommand{\1}{{\bf 1}}

\newcommand{\R}{\mathbb R}

\newcommand{\cb}{\mathcal B}
\newcommand{\cac}{\mathcal C}

\newcommand{\ce}{\mathcal E}
\newcommand{\cf}{\mathcal F}
\newcommand{\ch}{\mathcal H}

\newcommand{\cj}{\mathcal J}
\newcommand{\cl}{\mathcal L}

\newcommand{\cp}{{\mathcal P}}

\newcommand{\cs}{\mathcal S}

\newcommand{\al}{\alpha}
\newcommand{\ep}{\varepsilon}

\newcommand{\ga}{\gamma}

\newcommand{\om}{\omega}
\newcommand{\oom}{\Omega}

\newcommand{\si}{\sigma}
\newcommand{\te}{\theta}

\newcommand{\lp}{\left(}
\newcommand{\rp}{\right)}
\newcommand{\lc}{\left[}
\newcommand{\rc}{\right]}
\newcommand{\lcl}{\left\{}
\newcommand{\rcl}{\right\}}
\newcommand{\lln}{\left|}
\newcommand{\rrn}{\right|}
\newcommand{\lla}{\left\langle}
\newcommand{\rra}{\right\rangle}

\date{}

\begin{document}
	\title[Trapezoid rules]{Convergence of trapezoid rule to rough integrals}
	\author
	[Y. Liu \and Z. Selk \and S. Tindel]
	{Yanghui Liu \and  Zachary Selk \and Samy Tindel}

\address{Yanghui Liu: Department of Mathematics, Baruch College, CUNY, One Bernard Baruch Way
(55 Lexington Ave. at 24th St), New York, NY 10010, United States}

\address{Zach Selk: Department of Mathematics, Purdue University, 150 N. University Street, West Lafayette, IN 47907, United States}

\address{Samy Tindel: Department of Mathematics, Purdue University, 150 N. University Street, West Lafayette, IN 47907, United States}


\thanks{S. Tindel is supported by the NSF grant  DMS-1952966.}



	\begin{abstract}
	Rough paths techniques give the ability to define solutions of stochastic differential equations driven by signals $X$ which are not semimartingales and whose $p$-variation is finite only for large values of $p$. In this context, rough integrals are usually Riemann-Stieltjes integrals with correction terms that are sometimes seen as unnatural. As opposed to those somewhat artificial correction terms, our endeavor in this note is to produce a trapezoid rule for rough integrals driven by general $d$-dimensional Gaussian processes. Namely we shall approximate a generic rough integral $\int y \, dX$ by Riemann sums avoiding the usual higher order correction terms, making the expression easier to work with and more natural. Our approximations apply to all controlled processes $y$ and to a wide range of Gaussian processes $X$ including  fractional Brownian motion with a Hurst parameter $H>1/4$.	As a corollary of the trapezoid rule,   we   also consider the convergence of a midpoint rule for integrals of the form $\int f(X)  dX$. 
	\end{abstract}
	
		\maketitle

{
\hypersetup{linkcolor=black}
 \tableofcontents 
}

\section{Introduction}
Inspired by the seminal series of papers \cite{Ch, Young},
rough paths were first introduced in \cite{Lyons} in 1998 to study differential equations of the form:
\begin{equation}\label{CDE}
	y_t=y_0+\sum_{j=1}^d \int_0^t V_j(y_s)dX_s^j ,
\end{equation}
where $V_j$ are smooth bounded vector fields, $X:[0,T]\to\mathbb{R}^d$ is a given function with finite $p$-variation, usually a stochastic process, and $y:[0,T]\to\mathbb R^d$ is what is being solved for. Even though the final goal of the rough paths theory is to solve differential systems of the form \eqref{CDE} driven by arbitrary noisy inputs, the main step in the approach can be reduced to a proper definition of stochastic integrals like $\int_0^t V_j(y_s)dX_s^j$ above. In order to discuss this kind of integral, we   consider a generic partition $\cp=\{0=t_0<t_1<\cdots<t_{n+1}=t\}$ of $[0,t]$ and the following Riemann sum:
\begin{equation}\label{eq:riemann-sum}
	\mathcal{R}_0^t(V_j(y),X^j):=\sum_{k=0}^d V_j(y_{t_k})\delta X_{t_kt_{k+1}}^j,
\end{equation}
where we define $\delta X_{st}:=X_t-X_s$. In a classical setting, one obviously expects $\mathcal{R}_0^t(V_j(y),X^j)$ to converge to $\int_0^t V_j(y_s)dX_s^j$  as the mesh of $\cp$ goes to 0. Let us recall what this Riemann sum convergence becomes in a rougher and/or stochastic context.

\begin{enumerate}[wide, labelwidth=!, labelindent=0pt, label={(\roman*)}]
\setlength\itemsep{.1in}

\item
When $X$ is a semimartingale on a filtered probability space $(\Omega, \mathcal F, \mathcal{F}_t,\mu)$, we can use techniques from It\^o calculus to take limits in \eqref{eq:riemann-sum}. For example, assuming that $X$ is a $L^{2}$ continuous martingale, then we define the stochastic integral
\begin{equation}\label{eq:limit-for-semimartingales}
	\int_0^t V_j(y_s)dX_s^j\stackrel{L^2(\Omega)}{=}\lim_{|\cp|\to 0}\mathcal{R}_0^t(V_j(y),X^j),
\end{equation}
where the limit of the Riemann sums is understood in the $L^2(\Omega)$ sense. The stochastic integral \eqref{eq:limit-for-semimartingales} can be extended to all processes  which are adapted to the filtration $\mathcal{F}_t$ and square integrable with respect to the bracket of $X$.

\item
When $X$ is not a semimartingale, but is assumed to have finite $p$-variation for  $p<2$, one can rely on the classical Young-Stieltjes integration to take limits in \eqref{eq:riemann-sum}. 
Then one invokes a result in \cite{Young} which can be summarized as follows.
\begin{proposition}\label{prop:young-integration}
	Let $f\in C^{p\text{-var}}([0,T],\mathbb R^d)$ and $g\in C^{q\text{-var}}([0,T],\mathbb R^d)$ with $1/p+1/q>1$. Then the Riemann-Stieltjes integral exists:
	\begin{equation}
		\int_0^t f_s dg_s=\lim_{|\cp|\to 0}\mathcal{R}_0^t(f,g) . 
	\end{equation}
	
\end{proposition}
\noindent
Proposition \ref{prop:young-integration} can be applied directly in order to analyze \eqref{eq:riemann-sum}. Specifically, if we assume that $y$ should have the same $p$-variation as $X$, then we can make sense of \eqref{CDE} in Young-Stieltjes sense whenever $X$ is a stochastic process with finite $p$-variation almost surely for $p<2$. In this case we define:
\begin{equation}
	\int_0^t V_j(y_s)dX_s^j\stackrel{\textrm{a.s.}}{=}\lim_{|\cp|\to 0}\mathcal{R}_0^t(V_j(y),X^{j}),
\end{equation} 
that is $\int_0^t V_j(y_s)dX_s^j$ is defined as an almost sure limit of Riemann sums.

\end{enumerate}

However, if we want to solve \eqref{CDE} beyond the semimartingale or the finite $p$-variation case with $p<2$, rough paths theory is the main available path-wise type method. We give a brief introduction to this method in Section \ref{sec:rough-path-above-X}, and refer to \cite{FrizBook,Lyons,HairerBook} for a more comprehensive guide. At this point we just mention that rough path theory lets us solve \eqref{CDE} for $X\in C^{p\text{-var}}$ and for arbitrary $p$, provided we can define the following stack of iterated integrals of order  $n=1,2,\dots, \lfloor p\rfloor$:
\begin{align}\label{eq:def-iterated-intg}
&X_{st}^{1,i}=\int_s^t dX_r^i, \quad
	X_{st}^{2,ij}=\int_s^t\int_s^r dX_u^i dX_r^j, \notag\\
	&X_{st}^{3,ijk}=\int_s^t\int_s^r\int_s^u dX_v^i dX_u^j dX_r^k, \quad
	\cdots
\end{align}
Notice that in this paper the analysis is restricted to the case $p<4$, which corresponds to needing to define the first three integrals in \eqref{eq:def-iterated-intg}. This is due to issues in defining these integrals for worse $p$-variation. As mentioned above, once the iterated integrals in~\eqref{eq:def-iterated-intg} are properly defined, we can solve~\eqref{CDE} thanks to the rough paths machinery whenever $X\in C^{p\text{-var}}$. In particular  if $y$ is the solution to~\eqref{CDE},
one can define the stochastic integral $\int y\, dX$ as the following limit of modified Riemann sums:
\begin{equation}\label{Rough Integral}
	\int_0^t y_{s} \, dX_s=\lim_{n\to\infty}\sum_{k=0}^{n} y_{t_{k}} \, X^{1}_{t_k,t_{k+1}}
	+V (y_{t_k}) \, X_{t_kt_{k+1}}^{2}+ V'V(y_{t_k}) \, X^{3}_{t_kt_{k+1}} ,
\end{equation}
where we have considered a 1-dimensional situation (namely $d=1$) in order to avoid cumbersome indices.
Standard and relevant applications of rough paths techniques are the ability to define and solve stochastic differential equations driven by fractional Brownian motion or other processes with low regularity that are not semimartingales. Even for equations driven by usual Brownian motions, the continuity results related to rough paths techniques bring simplifications in classical stochastic analysis results such as large deviations principles, see \cite{LD}. Other relevant applications include  data analysis, see \cite{ChineseChar} and filtering theory, see \cite{RF}.

Nevertheless,   in spite of the rough paths theory's numerous  achievements,   the definition~\eqref{Rough Integral} for an integral of $y$ with respect to $X$ is sometimes seen as somehow not natural due to the higher order ``correction'' terms. In addition, the presence of the high order iterated integrals in the right hand side of~\eqref{Rough Integral} makes the rough integral approximation difficult to implement numerically. Ideally, one would thus like to take limits on simple Riemann sums like \eqref{eq:riemann-sum}.

The natural endeavor of approximating rough (or other generalized stochastic) integrals by suitable Riemann sums has 
 been mostly carried out in case of a 1-dimensional fractional Brownian motion $B$ and for integrals of the form $\int f(B)dB$.  The   contributions in this direction includes \cite{BNN, GNRV, HN, HLN, NN, NR, NRS,  RV}. We should also mention that the approximation of rough integrals like \eqref{Rough Integral} by trapezoid rules, for a 1-d fractional Brownian motion and a general controlled process $y$, has been considered in \cite{OneD}. Namely, the following approximation of a rough integral $\int y dB$ is proposed in \cite{OneD}:
\begin{equation}\label{eq:def-trapezoid-1d}
	\text{tr-} \cj_0^t(y ,B):=\sum_{k=0}^n \frac{ y_{t_k} + y_{t_{k+1}} }{2}\delta B_{t_k t_{k+1}},
\end{equation} 
where $B$ is a one-dimensional fractional Brownian motion with Hurst parameter $H\ge1/6$ and where $y$ is a process whose increments are controlled by $B$ (see Definition~\ref{def:ctrld-process} below for the notion of controlled process). For $H>1/6$ it is proven that this trapezoid rule converges to the rough integral \eqref{Rough Integral}, while in the case $H=1/6$ an additional Brownian term pops out in the limit and the convergence holds   in the weak sense only. Notice that this phenomenon had already been observed in \cite{HN,NR, NRS} for integrands of the form $y_{t}=f(B_{t})$ for a sufficiently smooth function $f$.

\noindent With those preliminary remarks in mind, the main aim of the current contribution is to extend the scope of trapezoid type approximations to rough integrals. Our generalizations will go in two directions, that is \emph{(i)} we shall prove the convergence of trapezoid rules for $d$-dimensional Gaussian processes and \emph{(ii)} we handle the case of a general class of Gaussian processes beyond the fractional Brownian case. Our prototype of convergence theorem is stated below in an informal way. The reader is referred to Theorem \ref{theorem:Trapezoid-Rule} for a more precise statement.
\begin{theorem}\label{thm:cvgce-trapezoid-intro}
	Let $X$ be a centered Gaussian process on $[0,T]$   admitting a sufficiently regular covariance function $R$ in the 2-d $\rho$-variation sense. Denote the rough path lift of $X$ by $\textbf{X}=(X^1,X^2,X^3)$. Let $\mathbf{y}=(y,y^{1}, y^{2})$ be a process controlled by $\textbf{X}$ (examples of controlled processes include $y_{t}=f(X_{t})$ and solutions of SDEs driven by $\textbf{X}$). For a given partition of $[0,T]$, $\cp=\{0=t_0<t_1<\cdots<t_{n+1}<T\}$, we define the trapezoidal rule:
	\begin{equation}\label{eq:trap}
	\operatorname{tr-} \cj_0^T(y,X)= \sum_{k=0}^n\frac{y_{t_k}  +y_{t_{k+1}}  }{2}   \, X_{t_k t_{k+1}}^{1  } .
	\end{equation}
	Then as the mesh size $|\cp|\to 0$ we have 
	$$
	\operatorname{tr-} \cj_0^T(y,X)\xrightarrow[\text{a.s.}]{}\int_0^T  {y}_s  \, d\mathbf{X}_s ,
	$$ 
where the right hand side above designates the rough integral of $y$ against $X$.
\end{theorem}

\noindent 
As mentioned above, our main Theorem \ref{thm:cvgce-trapezoid-intro} shows that one can approximate rough integrals by very natural Riemann sums, for a wide class of integrands $y$ and Gaussian driving noises~$X$.
As a corollary of our trapezoid rule, we will also     prove a midpoint rule for rough integrals of the form $\int f(X)  dX$; see Corollary \ref{cor.mid}.

 Let us   mention a few words about the techniques employed for our proofs. Indeed, most of the aforementioned 1-d contributions concerning convergences of trapezoid rules rely heavily on integration by parts techniques from Malliavin calculus, together with central limit theorems for random variables in a fixed chaos. The generalization described in Theorem~\ref{thm:cvgce-trapezoid-intro} requires a new set of methods. Specifically, we shall use a combination of rough paths techniques in discrete time in order to single out the main terms in~\eqref{eq:def-trapezoid-1d}. Then we can simplify the main part of the computations by performing our integration by parts on the building blocks of our rough path $\bxx$ only. Eventually we invoke some limit theorems for weighted sums in order to get our limit results.  

Here is a brief outline of our paper. In Section \ref{sec:preliminary-material} we set the ground for our computations by recalling some basic facts about rough paths analysis, Gaussian processes and Malliavin calculus.  Section~\ref{sec.3} is then devoted to the trapezoid rule. Namely  Section~\ref{sec:preliminary-results} gives some preliminary results about Young integrals and convergence of random sequences. Then some random sums in a finite chaos are analyzed in Section~\ref{subseq:lp-bounds}. The corresponding weighted sums are handled in Section~\ref{sec:bounds-weighted-sums}. With all those results in hand, our main theorem is proved in Section~\ref{sec:proof-main-thm}. Eventually we give a brief list of processes to which our general result applies in Section~\ref{sec:applications}.

 In this article, $C$ denotes a constant which may change from line to line. In the same way, $G$ will denote a generic integrable random variable whose value may change from line to line.

\section{Preliminary material}\label{sec:preliminary-material}

This section contains some basic tools from Malliavin calculus and rough paths theory, as well as some analytical results which are crucial for the definition and integration of controlled processes. 

\subsection{Elements of rough paths}\label{sec:rough-path-above-X}

In this section we shall recall the notion of a rough path above a continuous path $X$, and how this applies to Gaussian processes. The interested reader is referred to \cite{HairerBook,FrizBook} for further details.

\subsubsection{Basic rough paths notions}\label{subsubsection.notions}
\noindent For $s<t$ and $m\geq 1$, consider the simplex $\cs_{m}([s,t])=\{(u_{1},\ldots,u_{m})\in\lbrack s,t]^{m};\,u_{1}<\cdots<u_{m}\} $.  We start by introducing the notion of increments, which turns out to be handy in the definition of a rough path.

\begin{definition}\label{def:increments}
Let $k\ge 1$. Then the space of $(k-1)$-increments, denoted by $\cac_k([0,T],\R^d)$ or simply $\cac_k(\R^d)$, is defined as
$$
\cac_k(\R^d)\equiv \lcl g\in C(\cs_{k}([0,T]); \R^n) ;\, \lim_{t_{i} \to t_{i+1}} g_{t_1 \cdots t_{k}} = 0, 
\text{ for all } i\le k-1 \rcl.
$$
\end{definition}

\noindent
In the sequel we will also often resort to a finite difference operator called $\delta$, which acts on increments and is useful to split iterated integrals into simpler pieces.
\begin{definition}\label{def:delta-on-C1-C2}
Let $g\in\cac_1(\R^d)$, $h\in\cac_2(\R^d)$. Then for $(s,u,t)\in\cs_{3}([0,T])$,  we set
\begin{equation*}
  \delta g_{st} = g_t - g_s,
\quad\mbox{ and }\quad
\delta h_{sut} = h_{st}-h_{su}-h_{ut}.
\end{equation*}
\end{definition}

\noindent
In order to define rough integrals, some minimal regularity assumptions on increments will have to be made. In particular, it will be convenient to measure 
the regularity of increments in $\cac_{1}$ and $\cac_{2}$ in terms of $p$-variation.
\begin{definition}\label{def:var-norms-on-C2}
For $f \in \cac_2(\R^d)$, $p>1$ we set
$$
\|f\|_{p{\rm \text{-var}}}
=
\|f\|_{p{\rm \text{-var}}; [0,T]}
=
\sup_{\cp \subset [0,T]}\left(\sum_{i} |f_{t_{i}t_{i+1}}|^p\right)^{1/p},
$$
where the supremum is taken over all subdivisions $\cp$ of $[0,T]$. The set of increments in $\cac_{2}(\R^n)$ with finite $p$-variation is denoted by $\cac_{2}^{p\text{-var}}(\R^d)$.  For $f \in  \cac_{1}(\RR^{d})$, we denote $\|f\|_{p\text{-var}} = \| \delta f\|_{p\text{-var}} $. 
\end{definition}
We will also make an extensive use of H\"older norms, whose definition is recalled below:  
\begin{definition}\label{def.2.4}
We denote by $\cac_{2}^{\gamma}(\R^d)$ the space of $\gamma$-H\"older functions on $[0,T]$. That is,
	\begin{equation}
		\cac_{2}^{\gamma}(\R^d)=\left\{f\in \cac_{2} (\R^d) : \sup_{s,t\in [0,T]} \frac{|  f_{st}|}{|t-s|^{\gamma}}<\infty\right\}. 
	\end{equation}
 We define  $\cac_{1}^\gamma(\R^d)$ the space of functions $f$ such that $\delta f \in \cac_{2}^\gamma(\R^d)$.
\end{definition}

\noindent With these preliminary definitions in hand, we   now define the notion of a rough path above a  continuous $p$-variation  path $x$ with   $p>1$. 

\begin{definition}\label{def:RP}
Let $x$ be a    continuous  $\R^d$-valued  $p$-variation  path   for some $p>1$. We say that $x$  gives rise to a geometric $p$-rough path if there exists a continuous path $(x^{n }_{st}    , \,
(s,t)\in\cs_{2}([0,T]) )$ with values in $(\RR^{d})^{\otimes n}$ for each $ n\le \lfloor  p \rfloor $ such that $x^{1}_{st}=\delta x_{st}$,
  and a control function $\omega_{x}$ (in the sequel a control will always stand for a two variables function  on $\cs_{2} ([0,T]) $ which satisfies super-additivity conditions) such that

\noindent
\emph{(1) Regularity:}
For all $n\le \lfloor  p \rfloor$,   $x^n$ satisfies  $|x^{n}_{st}| \leq \omega_{x} (s,t)^{n/p}$.

\noindent
\emph{(2) Multiplicativity:}
With $\delta x^n$ as in Definition \ref{def:delta-on-C1-C2}, we have
\begin{equation}\label{eq:multiplicativity}
\delta x^{n }_{sut}=\sum_{n_1=1}^{n-1}
x^{n_{1} }_{su}  \otimes 
x^{n-n_{1} }_{ut} .
\end{equation}

\noindent
\emph{(3) Geometricity:}
Let  $x^{\ep}$ be a sequence of piecewise linear approximations of $x$. For any $n\le \lfloor p \rfloor$  we assume that $x^{\ep,n }$ converges in $\frac p n$-variation  norm to $x^{n }$, where $x^{\ep,n }_{st}$ is defined for $(s,t)\in\Delta_{2}$ by
\begin{equation}\label{eq:geometricity}
x^{\ep,n }_{st}
=\int_{(u_{1},\ldots,u_{n})\in\cs_{n}([s,t])} dx_{u_{1}}^{\ep } \otimes  \cdots\otimes dx_{u_{n}}^{\ep }.
\end{equation}

\noindent
In the sequel we will write $\bx$ for the rough path above $x$, that is
\begin{equation*} {\bf x}_{st}=
\lp
x^{1 }_{st} , \dots, x^{\lfloor p\rfloor}_{st}
\rp, \quad (s,t)\in\cs_{2}([0,T]) .
\end{equation*}
\end{definition}

\noindent 
One of the key success factors of the rough path theory is its ability to give a proper definition of stochastic calculus in very general contexts. Within this framework, the generic integrands in stochastic type integrals are so-called controlled paths, whose definition is recalled below. We first introduce some necessary notations about matrix products.

\begin{notation}\label{not:tensor-products}
In order to avoid lengthy indices in our formulae throughout the paper, we will adopt the following convention for matrix products: two generic elements  $v \in (\R^{d})^{\otimes k} $ and $u\in \cl( (\R^{d})^{\otimes k}, \R^{m} )$ will stand for families
\begin{eqnarray*}
v &=& 
\lcl v^{  j_{1}\cdots j_{k} }; \,\,    j_{1}, \dots, j_{k}  \in \{1,\dots, d\} \rcl \\
u &=& 
\lcl u^{i j_{1}\cdots j_{k} } ;\,\, i \in \{1,\dots, m\},\,\,  j_{1}, \dots, j_{k}  \in \{1,\dots, d\} \rcl , 
\end{eqnarray*}
where $ v^{  j_{1}\cdots j_{k} } $ and $ u^{i j_{1}\cdots j_{k} } $ are real numbers. 
In this context the product $uv$ is defined as an element of $\R^{m}$ such that for $1\le i\le m$ we have
 \begin{align*}
(uv)^{i}: = \sum_{j_{1}, \dots, j_{k}=1}^{d}  
u^{i j_{1}\cdots j_{k} }   \times
v^{  j_{1}\cdots j_{k} }. 
\end{align*}
Similarly for  $1\le k' \le k$ and 
$w \in (\R^{d})^{\otimes k'}$, we define   $uw$ as an element of $\cl ( (\R^{d})^{\otimes (k-k')}, \R^{m} )$ such that for $1\le i\le m$ and $1\le j_{k'+1},\ldots,j_{k}\le d$ we have
\begin{align*}
(uw)^{i j_{k'+1}\cdots j_{k}}: = \sum_{j_{1}, \dots, j_{k'}=1}^{d}  
u^{i j_{1}\cdots j_{k'}j_{k'+1}\cdots j_{k} }\times
v^{  j_{1}\cdots j_{k'} }. 
\end{align*}
\end{notation}

We can now    state  the definition of controlled process in  the $p$-variation framework.

\begin{definition}\label{def:ctrld-process}
	Let ${\bf x}= (x^{1}, \dots, x^{\lfloor p \rfloor})$ be a   $p$-variation rough path as introduced in Definition~\ref{def:RP}. Let $y^0,\ldots,y^{\ell-1}$ be continuous processes $y^k:[0,T]\to\cl ((\mathbb R^d)^{\otimes k},\mathbb R^m)$ and define the remainder terms:
	\begin{equation}\label{eq:remainder-k}
	r_{st}^{k }=\delta y_{st}^{k }
	-y_s^{k+1 }x_{st}^{1 }-\cdots -y_s^{\ell-1 }x_{st}^{\ell-k-1 }
	\end{equation}
	and $r_{st}^{\ell-1 }=\delta y_{st}^{\ell-1 }$, where we recall the notation $\delta$ given in Definition~\ref{def:delta-on-C1-C2} and our Notation~\ref{not:tensor-products} on matrix products. If there is a control function $\omega_{y}$ such that  
	\begin{align*}
|\delta y^{k}_{st}| \leq \omega_{y}(s,t)^{1/p}
\qquad \text{and} \qquad
|r_{st}^k|
	\leq \omega_{y}(s,t)^{(\ell-k)/p}
\end{align*}
 for  all  $k=0,1,\dots, \ell-1$, then we say that $ {\bf y}=(y^0,\ldots ,y^{\ell-1})$ is a  $\R^{m}$-valued path of order $\ell$ controlled by ${\bf x}$.

\end{definition}
\begin{remark}
	A controlled path has to be seen as a continuous path whose increments are ``dominated'' by the increments of $x$. Namely, $y^0$ is a continuous path taking values in $\mathbb{R}^m$ such that the increments $\delta y_{st}^0$ as given  in Definition~\ref{def:delta-on-C1-C2} are given by
	\begin{equation}\label{eq:def-controlled-process}
		\delta y_{st}^{0 }=y_s^{1 }x_{st}^{1 }+\cdots +y_s^{\ell-1 }x_{st}^{\ell-1 }+r_{st}^{0 }
	\end{equation}
	The other relations in Definition \ref{def:ctrld-process} are imposed for algebraic sake.
\end{remark}
\begin{remark}
In this paper, our main integration results will concern controlled paths. Hence it is worth recalling that this class of processes is rich enough. It includes for instance 
	solutions of differential equations driven by $x$ such as \eqref{CDE}, as well as continuous paths of the form $g(x)$ for a smooth enough function $g$. 
\end{remark}

\begin{remark}\label{remark.holder}
We will  also use $\gamma$-H\"older norms version of Definitions \ref{def:RP} and \ref{def:ctrld-process} in our discussion.  We will omit these definitions for sake of conciseness. As an example, let us just mention that in a $\ga$-H\"older version of \eqref{eq:remainder-k} we would assume $y^{k}\in\cac_{1}^{\ga}$ and $r^{k}\in\cac_{2}^{(\ell-k)\ga}$.
\end{remark}

The following proposition contains the classical result about integration of controlled processes with respect to a rough path, together with an approximation of the integral by enriched Riemann type sums.
\begin{proposition}\label{prop:intg-ctrld-process}
	Let ${\bf x}$ be a 
continuous $p$-variation    
rough path   on $[0,T]$ and let $ {\bf y}$ be a $\R^{d}$-valued   path of order $\ell=\lfloor p \rfloor$ controlled by $ {\bf x}$ as introduced in Definition \ref{def:ctrld-process}. Consider a sequence of partitions of $[0,T]$ with mesh size $|\cp|\to 0$. Then the following limit:
\begin{equation}
\label{eqn.ri}
\lim_{|\cp|\to 0}  \sum_{k=0}^n y_{t_k}^{0 }x_{t_k t_{k+1}}^{1 }+y_{t_k}^{1 } x_{t_k t_{k+1}}^{2 }+\cdots+y_{t_k}^{\ell-1 }x_{t_k t_{k+1}}^{\ell }
\end{equation}
  exists almost surely. It is called   rough integral of $y$ with respect to ${\bf x}$ and is denoted by $\int_0^T  {y}_s   d {\bf x}_s$.\end{proposition}

\noindent 
One of the crucial ingredients in rough paths theory is the sewing lemma for integration. We label two discrete versions of this lemma, taken from \cite{Euler}, for further use. 
 In the following we denote $\cs_{m} (\llbracket s,t\rrbracket) = \{(u_{1},\ldots,u_{m})\in\llbracket s,t\rrbracket^{m};\,u_{1}<\cdots<u_{m}\} $, where $\llbracket s,t\rrbracket$ denotes the discrete interval related to a given partition of $[s,t]$ (see our forthcoming Notation~\ref{not:discrete-interval}).  For notational sake, we just write $\cs_{m}$ for $\cs_{m}(\ll 0, T\rr)$.

 \begin{lemma}\label{lem2.4}
 Consider  a Banach space $(\cb,\|\cdot\|)$ and $Q : \cs_{2}    \to \cb $. Recall that we set $\delta Q_{sut} = Q_{st}-Q_{su}-Q_{ut}$. 
Suppose that $\omega $ is a control on $ \ll 0, T\rr$.    
Moreover, assume that $Q_{t_{k}t_{k+1}}=0$ for all $t_{k}\in \ll0,T\rr$ and  that  
  \begin{align}\label{eqn.dQ}
\|\delta Q_{sut}\|\leq \omega (s, t)^{\mu}
\end{align}
   for all $(s, u, t)\in \cs_{3} $. 
     Then the following relation holds for all $(s,t)\in \cs_{2}$:
 \begin{eqnarray*}
 \|Q_{st} \|  \leq K_{\mu} \, \omega(s, t)^{\mu} \,,
\quad\text{where}\quad
K_{\mu} = 2^{\mu} \, \sum_{l=1}^{\infty} l^{-\mu}.
\end{eqnarray*}
\end{lemma}

The following lemma is a particular case of  Lemma \ref{lem2.4} with $\omega(s,t)=|t-s|$, which can be seen as the sewing lemma in H\"older norm. 
\begin{lemma}\label{prop:sewing}   Fix a constant $\mu>1$. 
Let $Q$ be   as in Lemma \ref{lem2.4}, and we further assume that   there exists a constant $C>0$ such that 
  \begin{align*}
|\delta Q_{sut}|\leq C \cdot |t-s|^{\mu} ,
\qquad\text{for all } (s,u,t)\in \cs_{3}.
\end{align*}
Then   for all $(s,t)\in \cs_{2}$ we have
\begin{equation}
	|Q_{st}|  \leq CK_\mu |t-s|^{\mu} .
\end{equation}
\end{lemma}

\subsubsection{Gaussian processes as rough paths}

\noindent Let us now turn to a more probabilistic setting for our computations. Namely  we assume that $X_t=(X_t^1,\ldots,X_t^d)$ is a continuous, centered Gaussian process with i.i.d. components, defined on a complete probability space $(\Omega, \cf, \bp)$. The covariance function of $X$ is defined as follows
\begin{equation}\label{eq:def-covariance-X}
R(s,t):=E\lc X_{s}^{j} X_{t}^{j}\rc, 
\end{equation}
where $X^{j}$ is any of the components of $X$. We shall also resort to  the following notation in the sequel
\begin{equation}\label{eq:def-variance-Xt}
\si_{t}^{2} := E\lc  \lp  X_{t}^{j} \rp^{2} \rc,
\quad\text{and}\quad
\si_{st}^{2} := E\lc  \lp  \delta X_{st}^{j} \rp^{2} \rc .
\end{equation}
The information concerning $X$ used below is mostly encoded in the rectangular increments of the covariance function $R$, which are given for $s,t,u,v\in [0,T]$ by
\begin{align}\label{eq:rect-increment-cov-fct}
R^{st}_{uv} := R(t,v)-R(t,u)-R(s, v)+R(s,u).
\end{align}
Notice that whenever the function $R$  is given as a covariance function like in \eqref{eq:def-covariance-X},   the rectangular increments of $R$ can also be written as
\begin{equation}\label{eqn.rcov}
R^{st}_{uv}
=
E\lc (X_t^j-X_s^j) \, (X_v^j-X_u^j) \rc.
\end{equation}
Related to rectangular 2-d increments, the notion of 2-dimensional $\rho$-variation leads to an efficient way of constructing rough paths above a Gaussian process $X$. It will also feature prominently in our considerations, and thus we label its definition for further use.
\begin{definition}\label{def:mixed-variation}
For a general continuous function $R:[0,T]^{2}\to\R$ and a parameter $\rho\geq1$, we set 
\begin{align} 
\|R\|_{\rho-\va;[s,t]\times[u,v]}
:=\sup_{\substack{(t_{i})\in\mathcal{P}([s,t])\\
(t_{j}^{\prime})\in\mathcal{P}\left(\left[u,v\right]\right)
}
}
\left(\sum_{i,j }
\left|R_{t_{i}t_{i+1}}^{t_{j}^{\prime}t_{j+1}^{\prime}}\right|^{\rho}\right)^{\frac{1}{\rho}},\label{eq:mixed_var}
\end{align}
where $\mathcal{P}([s,t])$ denotes the set of all partitions of $[s,t]$ and where $R^{t_{j'}t_{j'+1}}_{t_i t_{i+1}}$ is defined in \eqref{eq:rect-increment-cov-fct}. 
\end{definition}

\noindent
We also define the space of functions in the plane with finite 2-d $\rho$-variation:
\begin{definition}\label{2d-rho-var-space}
	Given a finite dimensional vector space $E$, we define $C^{\rho\text{-var}}([0,T]^2,E)$ to be the space of all functions $f:[0,T]^2\to E$ such that $\|f\|_{\rho\text{-var}}<\infty$. 
\end{definition}

The standard assumption allowing to build a rough path above a generic Gaussian process concerns the $\rho$-variation of its covariance function. This is why we   assume that the following hypothesis holds throughout the paper.
     
\begin{hyp}\label{hyp.x}
 We assume that $X$ is a centered continuous Gaussian process with covariance function $R$ such that $\|R\|_{\rho\text{-var}}<\infty$ for $\rho\in [1,2)$. 
\end{hyp}

\begin{remark}\label{remark.rho}
Note that  the $\rho$-variation norm $\|\cdot\|_{\rho\text{-var}}$ introduced in Definition \ref{def:mixed-variation} uses grid-like partitions. 
As pointed out in  \cite{FV}, those $\rho$-variations do not enjoy super-additivity properties.   A standard way to circumvent  this problem is to replace the $\rho$-variation of Definition \ref{def:mixed-variation} by the so-called controlled 2-d $\rho$-variation norm (see \cite[Definition 1]{FV}), which we denote by~$|\cdot|_{\rho\text{-var}}$. The norm $|\cdot|_{\rho\text{-var}}$ does satisfy sup-additivity properties (cf \cite[Theorem 1 (iii)]{FV}). 
Therefore, although we have not assumed $ |R|_{\rho\text{-var}}<\infty$ in Hypothesis \ref{hyp.x}, we can consider a $\rho' = (\rho+\ep)>\rho$ for $\ep$ arbitrarily small  such that   $|R|_{\rho'\text{-var}}<\infty$ (this is ensured by \cite[relation (1.2)]{FV}). 
 Then one is allowed to pick the control $\om_{R}(D)=|R|_{\rho'\text{-var}}^{\rho'}(D)$. This kind of manipulation is routinely performed in  e.g \cite{JM,GOT}.
\end{remark}
With Remark \ref{remark.rho} in mind and  for notational sake, throughout the section we will skip the replacement of $\rho$ by $\rho'$ and pretend that the following is a consequence of    Hypothesis~\ref{hyp.x}. 


  \begin{hyp}\label{hyp.w}
  We assume that $X$ is a centered continuous Gaussian process with covariance function $R$, $\rho$ is a parameter lying in  $[1,2)$ and $\omega_{R}$ is a 2-d  control  such that for any rectangle $D $ we have
   \begin{align*}
\|R\|_{\rho\text{-var}, D}\leq \omega_{R}(D)^{1/\rho}.
\end{align*}
Recall that   we say that  $\omega_{R}$ is a 2-d control if it is  continuous, zero on degenerate rectangles and satisfies  
$ \omega_{R} (D) \geq \sum_{i=1}^n \omega_{R} (D_i)$ 
if $ D_i:0\le i\le n $ are disjoint  rectangles such that  $D=\cup_{i} D_{i}$.
  \end{hyp}
 
 The following result (stated e.g. in \cite[Theorem 15.33]{FV}) relates the 2-d $\rho$-variation of $R$ with the pathwise assumptions allowing to apply the abstract rough paths theory.

\begin{proposition}\label{prop:Gaussian-rough-path}
Let $X=(X^1,\ldots,X^d)$ be a continuous centered Gaussian process with  i.i.d. components and covariance function $R$ defined by \eqref{eq:def-covariance-X}. If $R$ satisfies Hypothesis \ref{hyp.x} then $X$ gives raise to a geometric $p$-rough path according to Definition~\ref{def:RP},  provided $p>2\rho$. 
\end{proposition}


\subsection{Wiener spaces associated to general Gaussian processes}\label{sec:wiener-space-general}

In this section we consider again the continuous, centered Gaussian process $X$ of Section~\ref{sec:rough-path-above-X}. Recall that its covariance function $R$ is defined by \eqref{eq:def-covariance-X}. We will describe the Cameron-Martin space assuming that we are in a real valued situation, the generalization to a $\R^{d}$-valued process being left to the patient reader.

The analysis of iterated integrals performed in Section 
\ref{sec.3}
 will be based on a Hilbert space $\ch$ allowing a proper definition of Wiener integrals as defined e.g in~\cite{NuaBook}. Namely   ${\mathcal{H}}$ is defined to be the
completion of the linear space of functions of the form
\[
\mathcal{E}
=
\left\{  \sum_{i=1}^{n}a_{i} \1_{\left[  0,t_{i}\right]  }:a_{i}\in%
\mathbb{R}
\text{, }t_{i}\in\left[  0,T\right]  \right\}  ,
\]
with respect to  the inner product%
\begin{equation}\label{eq:def-inner-pdt-H}
\left\langle \sum_{i=1}^{n} a_{i} \1_{[0,t_{i}]}  ,
\sum_{j=1}^{m}b_{j} \1_{[0,s_{j}]}  \right\rangle _{\mathcal{H}}
=
\sum_{i=1}^{n}\sum_{j=1}^{m}a_{i}b_{j}R\left(  t_{i},s_{j}\right) .
\end{equation}

\begin{remark}\label{representation H norm}
Consider the special case $X_0=0$, which means in particular that $R(0,0)=0$. Then, as suggested by \eqref{eq:def-inner-pdt-H}, for any $h_1, h_2\in\mathcal{H}$, we can infer that
\begin{align}\label{rep H norm}
\langle h_1, h_2\rangle_\mathcal{H}=\int_0^T\int_0^Th_1(s)h_2(t)dR(s,t),\end{align}
whenever the 2-d Young's integral on the right-hand side is well-defined (one can refer e.g to~\cite{FrizBook} for more details).
\end{remark}

Since $\mathcal{H}$ is the completion of $\ce$ with respect  to the inner product defined by~\eqref{eq:def-inner-pdt-H}, it is isometric to the Hilbert space
$H^{1}(  X)  \subseteq L^{2}(  \Omega,\mathcal{F},\bp)  $
which is defined to be the $\vert \cdot\vert _{L^{2}(\Omega)  }$-closure of the set
\[
\left\{  \sum\nolimits_{i=1}^{n}a_{i}X_{t_{i}}:a_{i}\in%
\mathbb{R}
,\text{ }t_{i}\in\left[  0,T\right]  ,\text{ }n\in%
\mathbb{N}
\right\}  .
\]
In particular, we have that $\vert \1_{\left[  0,t\right]  }\vert
_{\mathcal{H}}=\vert X_{t}\vert _{L^{2}\left(
\Omega\right)}$. The isometry between $\ch$ and $H^{1}\left(  X\right)$ is denoted by $X(h)$, and is called a Wiener integral.

\begin{remark}\label{rmk:H-on-subinterval}
As mentioned above in \eqref{eq:def-inner-pdt-H}, the space $\ch$ is a closure of indicator functions. Hence it can be defined on any interval $[a,b]\subset[0,T]$. We denote by $\ch([a,b])$ this restriction. For $[a,b]\subset[0,T]$, one can then check the following identity by a limiting procedure on simple functions
\begin{equation}\label{eq:norm-H-as-2d-young}
\lla f \, \1_{[a,b]} , \, g \, \1_{[a,b]} \rra_{\ch}
=
\lla f  , \, g  \rra_{\ch([a,b])}.
\end{equation}
\end{remark}

\subsection{Malliavin calculus for Gaussian processes}

In this section we review some basic aspects of Malliavin calculus. The reader is referred to~\cite{NuaBook} for further details.

 As in Section~\ref{sec:wiener-space-general}, the family
  $X_t=(X_t^1,\ldots,X_t^d)$ designates a continuous, centered Gaussian process with i.i.d.\ components, defined on a complete probability space $(\Omega, \cf, \bp)$. For sake of simplicity, we assume that $\cf$ is generated by $\{X_{t}; \, t\in[0,T]\}$. An $\mathcal{F}$-measurable real-valued random variable $F$ is said to be cylindrical if it can be
written, for some $m\ge 1$, as
\begin{equation*}
F=f\lp  X_{t_1},\ldots,X_{t_m}\rp,
\quad\mbox{for}\quad
0\le t_1<\cdots<t_m \le 1,
\end{equation*}
where $f:\mathbb{R}^m \rightarrow \mathbb{R}$ is a $C_b^{\infty}$ function. The set of cylindrical random variables is denoted by~$ {S}$. 

\smallskip

\noindent The Malliavin derivative is defined as follows: for $F \in  {S}$, the derivative of $F$ in the direction $h\in\ch$ is given by
\[
\mathbf{D}_h F=\sum_{i=1}^{m}  \frac{\partial f}{\partial
x_i} \left( X_{t_1},\ldots,X_{t_m}  \right) \, \langle h,  \mathbf{1}_{[0,t_i]}\rangle_{\ch}.
\]
More generally, we can introduce iterated derivatives. Namely, if $F \in
 {S}$, we set
\[
\mathbf{D}^k_{h_1,\ldots,h_k} F = \mathbf{D}_{h_1} \cdots\mathbf{D}_{h_k} F.
\]
For any $p \geq 1$, it can be checked that the operator $\mathbf{D}^k$ is closable from
$ {S}$ into $\mathbf{L}^p(\oom;\ch^{\otimes k})$. We denote by
$\mathbb{D}^{k,p}(\ch)$ the closure of the class of
cylindrical random variables with respect to the norm
\[
\left\| F\right\| _{k,p}=\left(  {E}\left[|F|^{p}\right]
+\sum_{j=1}^k  {E}\left[ \left\| \mathbf{D}^j F\right\|
_{\ch^{\otimes j}}^{p}\right] \right) ^{\frac{1}{p}},
\]
and we also set $\mathbb{D}^{\infty}(\ch)=\cap_{p \geq 1} \cap_{k\geq 1} \mathbb{D}^{k,p}(\ch)$. The divergence operator $\delta^{\diamond}$ is then defined to be the adjoint operator of $\mathbf{D}$. Namely, 
for a process $u = \{u_{t}; t\in [0,T]\}$ in the domain of $\delta^{\diamond}$ we have 
\begin{align}\label{eqn.ibp}
E[ \delta^{\diamond}(u) F ] = E[ \langle \mathbf{D}F, u\rangle_{\ch} ],
\end{align}
for all $F\in \mathbb{D}^{1,2}$. Notice that if $u \in \mathbb{D}^{1,2} (\ch)$, then we also have $u\in \dom( \delta^{\diamond}) $. A typical elementary increment which can be represented thanks to the divergence operator is the following: for $(s,t)\in\cs_{2}([0,T])$ and $1\le i \le d$ we have 
\begin{align}\label{eqn.dx}
\delta X^{i}_{st} = \delta^{\diamond} ( \mathbf{1}_{[s,t]} e_{i} ),
\end{align}
 where $e_{i}$ denotes the $i$-th element of the canonical basis in $\R^{d}$.

We close this section by recalling the following result on Hermite polynomials:

\begin{proposition}\label{Fin-Dim-Gaussian}
Let $X,Y$ be jointly normal random variables with $X,Y\sim \mathcal N(0,1)$ and denote by $H_n$ the $n$th Hermite polynomial.  Then the following holds true:
\begin{equation}
		E[H_n(X) H_m(Y)]=\begin{cases}
		0&\text{ if } n\neq m\\
		\frac{1}{n!}\lp E\lc XY \rc\rp^{n}&\text{ if } n=m .
		\end{cases}
	\end{equation}
\end{proposition}

\section{The trapezoid rule}\label{sec.3}

This section is devoted to a complete statement and proof of the informal Theorem~\ref{thm:cvgce-trapezoid-intro}. We will first analyze some discrete sums in a finite chaos related to our rough path $\mathbf{X}$ in Section~\ref{subseq:lp-bounds}, then move to some useful weighted sums in Section~\ref{sec:bounds-weighted-sums}. Eventually the main part of our proof will be achieved in Section~\ref{sec:proof-main-thm}. 

Throughout the section we consider a general centered Gaussian process $X$ which satisfies  Hypothesis \ref{hyp.w} . In particular, the covariance function $R$ is defined by \eqref{eq:def-covariance-X} and the variance of the increments $\delta X_{st}^i$ is denoted by $\sigma^2 (s,t)$ (see our notation~\eqref{eq:def-variance-Xt}). As in Section \ref{sec:preliminary-material}, $X$ admits a rough path lift $\mathbf{X}$.
We also label a notation which will be useful for our future computations. 

\begin{notation}\label{notation.dk}
Let $[s,t] \times [u,v]$ be a generic rectangle in $[0,T]^{2}$. Consider a grid-like partition $\cp = \{ [t_{k}, t_{k+1}] \times [\tilde{t}_{k'}, \tilde{t}_{k'+1}]; s=t_{0}<\cdots< t_{m}=t, u=\tilde{t}_{0}<\cdots <\tilde{t}_{n} =v \}$. Then we set $D_{kk'} = [t_{k}, t_{k+1}] \times [ \tilde{t}_{k'}, \tilde{t}_{k'+1} ] $. 
\end{notation}

\subsection{Two inequalities}\label{sec:preliminary-results}

We first derive an inequality on 2-d Young integrals   which we make extensive use of. It is an elaboration of the Young-Loeve-Towghi inequality \cite[Theorem 1.2]{To}. Recall that for $y\in C([0,T]^{2})$,   $y^{su}_{tv}$ denotes the rectangular increment of $y$ over $[s,t]\times[u,v]$   defined in~\eqref{eq:rect-increment-cov-fct}. 
\begin{lemma}\label{lemma.YLT}
	Let $z\in C^{\rho\text{-var}}([0,T]^2,\mathbb{R}^d)$ and consider a function $y$ sitting in the space 
	$C^{\theta\text{-var}}([0,T]^2, \cl(\mathbb{R}^d,\mathbb{R}^d))$ with $1/\rho+1/\theta>1$, where $C^{\rho\text{-var}}$ and $C^{\theta\text{-var}}$ are given in Definition \ref{2d-rho-var-space}. For some given  $s<t<\si$ and $u<v<\eta$ we set $D=[s,\si]\times [u,\eta]$ and   
	\begin{align}\label{eqn.yh}
\hat{y}_{tv} := \int_{[s,t]\times[u,v]} y^{sr }_{ur'}dz_{r r'}  ,
\end{align}
 Then   the $\rho$-variation of $\hat{y}$ on $D$ can be bounded as follows:
	\begin{equation}\label{eq:YLT}
		\left\| \hat{y} \right\|_{\rho\text{-var}, D}\le C\cdot\|y \|_{\theta\text{-var};D}\cdot \|z\|_{\rho\text{-var};D}.
	\end{equation} 
\end{lemma}
\begin{proof}
We consider partitions $s=t_{0}<\cdots<t_{m}=\si$ and $u=\tilde{t}_{0}<\cdots<\tilde{t}_{n}=\eta$ of $[s,\si]$ and $[u, \eta]$, respectively.  Recall from Notation \ref{notation.dk} that  we denote $D_{kk'}=[t_{k}, t_{k+1}]\times [\tilde{t}_{k'}, \tilde{t}_{k'+1}] \subset D$, and write 
\begin{align*}
\hat{y}(D_{kk'}) = \hat{y}_{\tilde{t}_{k'}\tilde{t}_{k'+1}}^{t_{k}t_{k+1}} = \hat{y}_{t_{k+1}\tilde{t}_{k'+1}}-\hat{y}_{t_{k+1}\tilde{t}_{k'}}-\hat{y}_{t_{k}\tilde{t}_{k'+1}}+\hat{y}_{t_{k}\tilde{t}_{k'}}. 
\end{align*}
Using expression \eqref{eqn.yh}, one can decompose $\hat{y}(D_{kk'})$ according to our partition in the following way: 
\begin{eqnarray}\label{eqn.i}
\hat{y}(D_{kk'}) &=& \int_{D_{kk'}} y_{ur'}^{sr}dz_{rr'}
\nonumber
\\
&=&  \int_{D_{kk'}} y_{\tilde{t}_{k'}r'}^{t_{k}r}dz_{rr'}+\int_{D_{kk'}} y_{\tilde{t}_{k'}r'}^{st_{k}}dz_{rr'}+\int_{D_{kk'}} y_{u\tilde{t}_{k'}}^{t_{k}r}dz_{rr'}+\int_{D_{kk'}} y_{u\tilde{t}_{k'}}^{st_{k}}dz_{rr'}
\nonumber
\\
&:=& I_{1}+ I_{2}+ I_{3}+ I_{4} . 
\end{eqnarray}

 
It should be noticed that the term $I_{1}$ above can be bounded directly thanks to the Young-Loeve-Towghi inequality \cite[Theorem 1.2]{To}. We get
\begin{align}\label{eqn.i1}
|I_{1}| \leq C\|y\|_{\theta\text{-var}, D_{kk'}} \cdot \|z\|_{\rho\text{-var}, D_{kk'}} .
\end{align}
The term $I_{4}$ can also be treated easily. Indeed, we have 
\begin{align*}
I_{4} = y^{st_{k}}_{u\tilde{t}_{k'}}\cdot z^{t_{k}t_{k+1}}_{\tilde{t}_{k}\tilde{t}_{k+1}},
\end{align*}
and thus
\begin{align}\label{eqn.i4}
|I_{4}| \leq   \|y \|_{\theta\text{-var} , [u,\tilde{t}_{k'}]\times [s, t_{k} ]}   \cdot \|z\|_{\rho\text{-var}, D_{kk'}} . 
\end{align}

We now focus on the 1-d type integral $I_{2}$ in equation \eqref{eqn.i}. In order to bound this term one can use the classical Young inequality \cite{Young} in order to get  
\begin{eqnarray}\label{eqn.i2}
|I_{2}|  &=& 
\left| \int_{[\tilde{t}_{k'}, \tilde{t}_{k'+1}]} y^{st_{k}}_{\tilde{t}_{k'}r'} d (z_{t_{k+1}r'} - z_{t_{k}r'})   \right| 
\nonumber
\\
&\leq& 
C \, \|y^{st_{k}}_{\tilde{t}_{k'} \cdot} \|_{\theta\text{-var} ,   [\tilde{t}_{k'}, \tilde{t}_{k'+1}]}   
\cdot \|z_{t_{k+1} \cdot} - z_{t_{k} \cdot} \|_{\rho\text{-var}, [ \tilde{t}_{k'}, \tilde{t}_{k'+1}]}
\nonumber
\\
&\leq&  
C \, \|y \|_{\theta\text{-var} , [s,t_{k}]\times [\tilde{t}_{k'}, \tilde{t}_{k'+1}]}   \cdot \|z\|_{\rho\text{-var}, D_{kk'}}. 
\end{eqnarray}
In the same way, we can also upper bound the term $I_{3}$ in \eqref{eqn.i} as 
\begin{align}\label{eqn.i3}
|I_{3}| \leq  C  \|y \|_{\theta\text{-var} , [u,\tilde{t}_{k'}]\times [t_{k}, t_{k+1}]}   \cdot \|z\|_{\rho\text{-var}, D_{kk'}}. 
\end{align}
Plugging \eqref{eqn.i1}-\eqref{eqn.i3} into \eqref{eqn.i}, we have thus obtained that 
\begin{align*}
|\hat{y}(D_{kk'})| \leq & C\Big\{
\|y\|_{\theta\text{-var}, D_{kk'}} 
+
   \|y \|_{\theta\text{-var} , [s,t_{k}]\times [\tilde{t}_{k'}, \tilde{t}_{k'+1}]}   
\\
&\qquad
+    \|y \|_{\theta\text{-var} , [u,\tilde{t}_{k'}]\times [t_{k}, t_{k+1}]}   
+
   \|y \|_{\theta\text{-var} , [u,\tilde{t}_{k'}]\times [s, t_{k} ]}   
\Big\}\cdot \|z\|_{\rho\text{-var}, D_{kk'}} .
\end{align*}
Therefore by trivial monotonicity properties of $\te$-variations, we arrive at 
 \begin{align}\label{eqn.yh1}
|\hat{y}(D_{kk'})|  
\leq C   \|y  \|_{\theta\text{-var} , D}    \cdot \|z\|_{\rho\text{-var}, D_{kk'}}. 
\end{align}
As  explained  in Remark \ref{remark.rho}, we will skip   the routine procedure  of   replacing  $\rho$ by $\rho'$ in order to apply super-additivity relations to 2-dimensional $\rho$-variations. Hence summing relation~\eqref{eqn.yh1} over $k,k'$ and invoking super-additivity, we finally prove the desired inequality ~\eqref{eq:YLT}. 
\end{proof}

 We close this section by giving a general convergence lemma for a sequence of stochastic processes. It is borrowed from \cite[Lemma 3.5]{Euler}.
\begin{lemma}\label{G existence}
	Let $\{z^{n}, n\in \mathbb{N}\}$ be a sequence of stochastic processes such that 
	$$
	\|\delta z^n_{st}\|_{L^{p}(\Omega)}\le C_p n^{-\alpha}(t-s)^\beta,
 	$$ 
	for all $p\ge 1$, where $K_p$ is a constant depending on $p$ and where we recall the notation $\delta$ given in Definition~\ref{def:delta-on-C1-C2}. Then for $0<\gamma<\beta$ and $\kappa>0$, we can find an integrable random variable $G_{\gamma,\kappa}$ independent of $n$ and admitting moments of any order, such that:
	\begin{equation}
	\|z^n\|_\gamma\le G_{\gamma, \kappa} \, n^{-\alpha+\kappa} .
	\end{equation}
\end{lemma}

\subsection{Upper-bounds for processes in a finite chaos}\label{subseq:lp-bounds}
\noindent With the notions of Section \ref{sec:preliminary-material} in hand, we now introduce a family of processes defined as sums of iterated integrals of $X$ which appear naturally in the analysis of the approximation~\eqref{eq:trap}. We start by proving a bound on sums of L\'evy area type processes which generalizes \cite{HLN2, HLN3, Euler}. We first label a notation for further use. 

\begin{notation}\label{not:discrete-interval}
Let $\cp=\{0=t_0<\cdots<t_{n}=T\}$ be a partition of $[0,T]$. Take $s,t\in [0,T]$. Then $\llbracket s,t\rrbracket\colon = \{t_k\in \cp \colon t_k \in [s,t]\}$. We denote $\cs_k(\llbracket s,t\rrbracket)=\{(t_1,\ldots,t_k)\in \cp \colon t_1 \leq \cdots\leq t_k\}$ as the discrete simplex.
\end{notation}

The main bound involving L\'evy area type objects is the following. 
\begin{lemma}\label{lemma:F-bound}
	Suppose that Hypothesis \ref{hyp.w} holds true for the $\R^{d}$-valued Gaussian process $X=(X^1,\ldots,$ $X^d)$   with covariance function $R$,   2-d control   $\omega_{R}$ and $\rho \in [1,2)$.    Consider a partition $\{t_k: 0 \leq k\leq n\}$ of $[0,T]$ and define the process $F$ on $\llbracket 0,T\rrbracket$ by
	\begin{equation}\label{F}
	F_t^{ij}=\begin{cases}
	\sum\limits_{0\leq t_k<t}  X_{t_k t_{k+1}}^{2,ij}&\text{ for }i\neq j
	\vspace{.4cm}\\
	\sum\limits_{0\leq t_k<t}X_{t_k t_{k+1}}^{2,ii}-E[X_{t_kt_{k+1}}^{2,ii}]&\text{ for }i=j  ,
	\end{cases}
	\end{equation}
	with the convention that $F_0^{ij}=0$. Then for any $p\geq 1$ there exists a strictly positive constant $C=C_{p}$ such that for all $s,t\in \cs_2  $ and $0\leq \ep\leq 2-\rho$ we have
	\begin{equation}\label{eq:F-bound}
		\|\delta F_{st}^{ij}\|_{p}^{2}\le C \max_{k,k'} 
	\omega_{R}(D_{kk'})^{\frac{ \ep}{\rho} }
	 \cdot  \omega_{R}( [s,t]^{2})^{\frac{2-\ep}{\rho}} ,
	\end{equation}
	where $\|\cdot\|_p$ denotes the $L^p(\Omega)$ norm and where the rectangle $D_{kk'}$ is defined in Notation \ref{notation.dk}.
\end{lemma}
\begin{proof}
	By hypercontractivity for random variables in the second chaos (see \cite[Theorem 1.4.1]{NuaBook}), we just need to consider $p=2$. Furthermore, when $i=j$ notice that
	\begin{equation}\label{eq:second-order}
		X_{t_kt_{k+1}}^{2,ii}=\int_{t_k}^{t_{k+1}}\int_{t_k}^{u_1}dX_{u_2}^idX_{u_1}^i=\frac{1}{2}\lp\delta X_{t_kt_{k+1}}^{i}\rp^2, 
	\end{equation}
	where a complete justification of \eqref{eq:second-order} is due to the Definition \ref{def:RP} and the fact that $\mathbf{X}$ is assumed to be geometric. Therefore one can recast the definition of $\delta F^{ii}_{t_{k}t_{k+1}}$ in~\eqref{F} as
		\begin{eqnarray}\label{eq:hermite}
	\delta F^{ii}_{t_{k}t_{k+1}} =	X_{t_k t_{k+1}}^{2,ii}-E[X_{t_kt_{k+1}}^{2,ii}]&=&\frac{1}{2}(X_{t_kt_{k+1}}^{1,i})^2-\frac{1}{2} E[(X_{t_kt_{k+1}}^{1,i})^2]
		\notag\\
		&=&
		\frac{1}{2}
		\si_{t_{k}t_{k+1}}^{2}
		H_2\left(\frac{X_{t_kt_{k+1}}^{1,i}}{\si_{t_{k}t_{k+1}}}\right),
	\end{eqnarray}
	where $H_2(x)=x^2-1$ is the second Hermite polynomial and where the notation $\si_{st}^{2}$ has been introduced in \eqref{eq:def-variance-Xt}. Then putting together relations~\eqref{F}, \eqref{eq:second-order} and \eqref{eq:hermite} we have
\begin{align*}
&		\| \delta F_{st}^{ii}\|_2^2= E\left[\sum_{k}\left(X_{t_k t_{k+1}}^{2,ii}-E[X_{t_kt_{k+1}}^{2,ii}]\right)
\sum_{k'}\left(X_{t_{k'} t_{k'+1}}^{2,ii}-E[X_{t_{k'} t_{k'+1}}^{2,ii}]\right)\right]\\
&=
\frac14
E\left[\sum_{k} \si^{2}_{t_{k}t_{k+1}}
H_2\left(\frac{X_{t_kt_{k+1}}^{1,i}}{\si_{t_{k} t_{k+1}}}\right)\sum_{k'}
\si^{2}_{t_{k'}t_{k'+1}}
H_2\left(\frac{X_{t_{k'}t_{k'+1}}^{1,i}}{\si_{t_{k'} t_{k'+1}}}\right)\right].
\end{align*}
Next, expanding the double sum in $k,k'$ above we get
\begin{equation}\label{a1}
\|\delta F_{st}^{ii}\|_2^2 
=\frac{1}{4}\sum_{k,k'}\si_{t_{k},t_{k+1}}^{2} \si_{t_{k'},t_{k'+1}}^{2}
E\left[H_2\left(\frac{X_{t_kt_{k+1}}^{1,i}}{\si_{t_{k},t_{k+1}}}\right)H_2\left(\frac{X_{t_{k'}t_{k'+1}}^{1,i}}{\si_{t_{k'},t_{k'+1}}}\right)\right] .
\end{equation}
Since each $X_{t_kt_{k+1}}^{1,i}$ is a Gaussian random variable, we can now apply Proposition~\ref{Fin-Dim-Gaussian} with $X=\si_{t_{k} t_{k+1}}^{-1}X_{t_kt_{k+1}}^{1,i}$ and 
$Y=\si_{t_{k'} t_{k'+1}}^{-1} X_{t_{k'}t_{k'+1}}^{1,i}$. This yields
\begin{eqnarray}\label{eq:ii-bound}
	\| \delta F_{st}^{ii}\|_2^2\notag&=&
	\frac{1}{8}\sum_{k,k'}\si_{t_{k},t_{k+1}}^{2} \si_{t_{k'},t_{k'+1}}^{2}
	\left(E\left[\frac{X_{t_kt_{k+1}}^{1,i} \, X_{t_{k'}t_{k'+1}}^{1,i}}{\si_{t_{k},t_{k+1}} \, \si_{t_{k'},t_{k'+1}}}
	\right]\right)^2\\
	&=&\frac{1}{8}\sum_{k,k'}\left(R^{t_kt_{k+1}}_{t_{k'}t_{k'+1}}\right)^2 .
\end{eqnarray}
Notice that the sum on the right hand side of \eqref{eq:ii-bound} is a sum over rectangles $D_{k,k'}=[t_k,t_{k+1}]$ $\times [t_{k'},t_{k'+1}]$. 
Since we have $|R_{t_{k'} t_{k'+1}}^{t_k t_{k+1}}|\leq \omega_{R}(D_{kk'})^{1/\rho} $ thanks to Hypothesis \ref{hyp.w}, we obtain
\begin{eqnarray}\label{eq:ii-bound-2a}
\|\delta F_{st}^{ii}\|_2^2&\le& \frac{1}{8}\sum_{k,k'}\omega_{R}(D_{kk'})^{2/\rho}
\\
\label{eq:ii-bound-2}
	&\le& \frac{1}{8} \max_{k,k'} 
	\omega_{R}(D_{kk'})^{\frac{ \ep}{\rho} }
	 \cdot \sum_{k,k'}\omega_{R}(D_{kk'})^{\frac{2-\ep}{\rho}} .
\end{eqnarray}
Now recall that in our statement we have chosen  $0\leq \ep\leq 2-\rho$, which yields $(2-\ep)/\rho>1$. Hence invoking the sup-additivity of $\omega_{R}$ and thus of $\omega_{R}^{(2-\ep)/\rho}$,  we arrive at
\begin{equation}\label{eq:Fii}
\|\delta F^{ii}_{st}\|_2^2 \le 
\frac{1}{8} \max_{k,k'} 
	\omega_{R}(D_{kk'})^{\frac{ \ep}{\rho} }
	 \cdot  \omega_{R}( [s,t]^{2})^{\frac{2-\ep}{\rho}} . 
\end{equation}
\noindent Putting together inequality \eqref{eq:Fii} and the aforementioned hypercontractivity argument, our claim \eqref{eq:F-bound} is proved for $i=j$. 

\noindent Let us now handle the case $i\neq j$. Similarly to what we did in \eqref{a1}, we compute
\begin{eqnarray}\label{eq:Fij-before}
	\|\delta F_{st}^{ij}\|_2^2\notag&=&\sum_{k,k'}
	E\lc X_{t_kt_{k+1}}^{2,ij}X_{t_{k'} t_{k'+1}}^{2,ij}\rc \\
	&=&\sum_{k,k'}E\left[\int_{t_k}^{t_{k+1}}X_{t_k r}^{1,i}dX_r^{1,j}\int_{t_{k'}}^{t_{k'+1}}X_{t_{k'} r'}^{1,i}dX_{r'}^{1,j}\right] .
\end{eqnarray}
In order to compute the right and side of \eqref{eq:Fij-before} we proceed as in \cite[p. 402]{FrizBook}. Namely, we consider a Gaussian regularization $X^\epsilon$ of $X$ whose rough path lift $\mathbf{X}^\epsilon$ also converges to $\mathbf{X}$. Let us write $R^\epsilon$ for the covariance function of the process $X^{\epsilon}$ and $F^\epsilon$ be the process $F$ defined similarly to \eqref{F} for the process $X^\epsilon$. A simple application of Fubini's theorem yields
\begin{eqnarray}\label{eq:Fij-regularized}
	\|\delta F_{st}^{\epsilon,ij}\|_2^2\notag&=&\sum_{k,k'}E\left[\int_{t_k}^{t_{k+1}}\int_{t_{k'}}^{t_{k'+1}}X_{t_{k'} r'}^{\epsilon,1,i}X_{t_k r}^{\epsilon,1,i}dX_r^{\epsilon,1,j} dX_{r'}^{\epsilon,1,j}\right]\\
	&=&\sum_{k,k'}\int_{D_{k,k'}}R^{\epsilon, t_k r}_{t_{k'}r'} \, dR^{\epsilon}(r,r'),
\end{eqnarray}
where we recall that $D_{k,k'}=[t_k,t_{k+1}]\times [t_{k'},t_{k'+1}]$. Taking limits in \eqref{eq:Fij-regularized} as $\epsilon\to 0$ we get
\begin{equation}\label{eq:Fij-after}
	\|\delta F^{ij}_{st}\|_{2}^{2}=
	\sum_{k,k'}\int_{D_{k,k'}}R^{t_k r}_{t_{k'}r'}  dR(r,r'). 
\end{equation} 
Then a direct application of  inequality \eqref{eq:YLT} yields 
 \begin{equation}\label{eq:F-f-bound}
	\|\delta F^{ij}_{st}\|_2^2\le C \sum_{k,k'}   \|R\|_{\rho\text{-var};D_{k,k'}}^{2}, 
\end{equation}%
  which in turn implies  relation \eqref{eq:ii-bound-2a} thanks to Hypothesis \ref{hyp.w}. Starting from \eqref{eq:F-f-bound}, we can thus conclude as we did for \eqref{eq:ii-bound-2} and \eqref{eq:Fii} in the case $i=j$. 
Our result \eqref{eq:F-bound}
is now shown for the case $i\neq j$, which concludes our proof. 
\end{proof}
\noindent We now state an elaboration of Lemma \ref{lemma:F-bound} for third-order integrals.

\begin{lemma}\label{lemma:g-bound}
Suppose that Hypothesis \ref{hyp.w} holds for the Gaussian process $X=(X^1,\ldots,X^d)$   with covariance function $R$, 2-d control $\omega_{R}$ and $\rho \in [1,2)$. 
For   $i,j,\ell\in \{1,\ldots,d\}$, $(s,t) \in \cs_2 $ and a generic partition $\{t_{i}; 0\leq k\leq n\}$ of $[0,T]$ we denote 
\begin{align}\label{eqn.g}
 \delta g_{st}^{ij\ell} =\sum_{s\leq t_k<t} X^{3,ij\ell}_{t_{k}t_{k+1}}.
\end{align}
 	Then for any $p\geq 2$ there exists a   positive constant $C=C_{\rho,p}$ such that for all $(s,t)\in \cs_2 $ we have
	\begin{equation}\label{eq:g-bound}
		\|\delta g_{st}^{ij\ell} \|_{L^{p}(\Omega)}^{2}
		\le
		 C   \sum_{k,k'}  
    \omega_{R}(D_{kk'})^{3/\rho}
    + C  \left|
   \sum_{k,k'}    
    R_{t_{k}t_{k+1}}^{t_{k}t_{k+1}}R_{t_{k'}t_{k'+1}}^{t_{k'}t_{k'+1}}R_{t_{k}t_{k+1}}^{t_{k'}t_{k'+1}} 
\right|,  
	\end{equation} 
	where we recall our Notation \ref{notation.dk} for $D_{kk'}$.   
\end{lemma}
\begin{proof}
	Along the same lines as Lemma \ref{lemma:F-bound}, by hypercontractivity of random variables in the third chaos, we just need to consider $p=2$. We focus on this case in the remainder of the proof.
	We split our considerations in  several steps. 
	
\noindent   \textbf{Step 1:   equal indices.}\quad  In this step,  we show that  \eqref{eq:g-bound} holds when the indices are equal: $i=j=\ell$.  We first note that  using geometricity similarly to what we have done  in~\eqref{eq:second-order}, we get  $X_{t_{k} t_{k+1}}^{3,iii}=\frac{1}{6}(\delta X_{t_kt_{k+1}}^i)^3$. 
We now expand the cubic power in terms of Hermite polynomials along the same lines as for \eqref{eq:hermite}. Namely, recall that $x^{3} = H_{3}(x)-3x$. 
Therefore, renormalizing $(\delta X_{t_{k}t_{k+1}}^{i})^{3}$ by $\si_{t_{k}t_{k+1}}$ (recall that $\si_{st}$ is defined by \eqref{eq:def-variance-Xt}), we get 
\begin{align}\label{eqn.zeta}
X^{3,iii}_{t_{k}t_{k+1}} = \frac16 \si^{3}_{t_{k}t_{k+1}} H_{3} (\si_{t_{k}t_{k+1}}^{-1} \delta X_{t_{k}t_{k+1}}^{i}) +\frac12 \si_{t_{k}t_{k+1}}^{2} \delta X_{t_{k}t_{k+1}}^{i}. 
\end{align}

Let us now go back to our expression \eqref{eqn.g} for $i=j=\ell$. By linearity of the expected value we have 
\begin{align}\label{eqn.gn}
\|\delta g_{st}^{iii} \|_2^2  =  \sum_{k,k'} E \lc
X_{t_k t_{k+1}}^{3,iii}X_{t_{k'} t_{k'+1}}^{3,iii}\rc . 
\end{align}
Plugging relation \eqref{eqn.zeta} into \eqref{eqn.gn}, invoking Proposition \ref{Fin-Dim-Gaussian} and recalling that $\si_{t_{k}t_{k+1}}^{2} = R^{t_{k}t_{k+1}}_{t_{k}t_{k+1}} $, we obtain:
\begin{eqnarray}\label{eq:iii-bound}
\|\delta g_{st}^{iii} \|_2^2 &=&
\notag
 \frac{1}{216}   \sum_{k,k'}  
    \lp R_{t_{k}t_{k+1}}^{t_{k'}t_{k'+1}}\rp^{3}  + \frac14  \sum_{k,k'}    R_{t_{k}t_{k+1}}^{t_{k}t_{k+1}}R_{t_{k'}t_{k'+1}}^{t_{k'}t_{k'+1}}R_{t_{k}t_{k+1}}^{t_{k'}t_{k'+1}} 
 \\
& \leq & C   \sum_{k,k'}  
   \omega_{R}(D_{kk'})^{3/\rho}
   + C  \left|
   \sum_{k,k'}   ( R_{t_{k}t_{k+1}}^{t_{k}t_{k+1}}R_{t_{k'}t_{k'+1}}^{t_{k'}t_{k'+1}}R_{t_{k}t_{k+1}}^{t_{k'}t_{k'+1}}) 
\right|
    ,
\end{eqnarray}
where the last inequality stems from Hypothesis \ref{hyp.w} and where we have used our Notation~\ref{notation.dk} for $D_{kk'}$. 
We have thus proved our claim \eqref{eq:g-bound} for $p=2$ and $i=j=\ell$. As mentioned above, the general case $p\geq 2$ follows by hypercontractivity. 

  \noindent   \textbf{Step 2:  three distinct indices.}   We   turn to the proof of \eqref{eq:g-bound} for   $i,j,\ell$ all  distinct. To this aim, we first note that a regularization procedure similar to the one which lead to \eqref{eq:Fij-after} yields
\begin{align*}
	\|\delta g_{st}^{ij\ell} \|_2^2&=\sum_{k,k'}E[X_{t_{k} t_{k+1}}^{3,ij\ell}X_{t_{k'} t_{k'+1}}^{3,ij\ell}]\\
	&=\sum_{k,k'}E\left[\int_{t_k}^{t_{k+1}}\int_{t_k}^{u}X_{t_k,v}^{1,i}dX_{v}^{1,j}dX_u^{1,\ell}\int_{t_{k'}}^{t_{k'+1}}\int_{t_{k'}}^{u'}X_{t_{k'},v'}^{1,i}dX_{v'}^{1,j}dX_{u'}^{1,\ell}\right]\\
	&=\sum_{k,k'}E\left[\int_{t_k}^{t_{k+1}}\int_{t_{k'}}^{t_{k'+1}}\int_{t_k}^u\int_{t_{k'}}^{u'}X_{t_k,v}^{1,i}X_{t_{k'},v'}^{1,i}dX_{v}^{1,j}dX_{v'}^{1,j}dX_{u}^{1,\ell}dX_{u'}^{1,\ell}\right],
\end{align*}
	for all $(s,t)\in \cs_2(\llbracket 0,T\rrbracket)$. Then using independence and Fubini's theorem (here again the standard  regularization arguments are outlined in \eqref{eq:Fij-after}) we get
	\begin{align}\label{eqn.49}
		\|\delta g_{st}^{ij\ell} \|_2^2=\sum_{k,k'}\int_{t_k}^{t_{k+1}}\int_{t_{k'}}^{t_{k'+1}}\int_{t_k}^u\int_{t_{k'}}^{u'}R^{t_k v}_{t_{k'} v'} dR(v,v') dR(u,u').
	\end{align}
Now applying twice inequality \eqref{eq:YLT} to the right side of \eqref{eqn.49}, we easily get 
\begin{equation}\label{eqn.51}
		\|\delta g_{st}^{ij\ell} \|_2^2\le C \sum_{k,k'} \omega_{R}(D_{kk'})^{3/\rho}
    .
\end{equation}
Our claim \eqref{eq:g-bound} is now proved for $i,j,\ell$ distinct.

 \noindent   \textbf{Step 3: two distinct indices.} 	  We now turn to the case of two distinct indices in $i,j,\ell$. In fact we should divide this case into $3$ distinct subcases, namely $i=j\neq \ell$, $i\neq j= \ell$ and $i=\ell\neq j$. We will only treat the first case $i=j\neq \ell$, the other ones being similar and left to the patient reader.
	
	  
	  We proceed as previously, resorting to the geometric property of $\mathbf{X}$, some regularization arguments, and Fubini's theorem. Recalling Notation \ref{notation.dk} for the intervals $   [t_{k}, t_{k+1}] \times [ \tilde{t}_{k'}, \tilde{t}_{k'+1} ] $, we get  
	\begin{eqnarray*}	\|\delta g_{st}^{ii\ell} \|_2^2 &=&\sum_{k,k'}E\left[X_{t_{k} t_{k+1}}^{3,ii\ell}X_{t_{k'} t_{k'+1}}^{3,ii\ell}\right]\\
	&=&\frac{1}{4}\sum_{k,k'}\int_{D_{kk'}} E\lc(X_{t_k,u}^{1,i})^2(X_{t_{k'},u'}^{1,i})^2\rc dR(u,u') .
	\end{eqnarray*}
In order to evaluate the term $E[(X_{t_k,u}^{1,i})^2(X_{t_{k'},u'}^{1,i})^2]$ above, we reproduce the steps leading from~\eqref{eq:hermite} to \eqref{eq:ii-bound} in the proof of Lemma \ref{lemma:F-bound} (based on Hermite polynomial decompositions). This yields 
	\begin{eqnarray}\label{eq:g-iil-bound-initial}
	\|\delta g_{st}^{ii\ell} \|_2^2  :=I_{1}+I_{2},
	\end{eqnarray}
	where the terms $I_{1}$ and $I_{2}$ are defined by 
	\begin{align}\label{eqn.I12}
I_{1} = \frac{1}{4}\sum_{k,k'}\int_{D_{kk'}} 
	R_{t_{k}u}^{t_{k}u}R_{t_{k'}u'}^{t_{k'}u'}
	 dR(u,u') 
	 \qquad \text{and}
\qquad	 I_{2} = \frac{1}{4}\sum_{k,k'}\int_{D_{kk'}} 
(R_{t_{k}u}^{t_{k'}u'})^{2}
	dR(u,u') . 
\end{align}
In the following we bound the two quantities $I_{1}$ and $I_{2}$. 

In order to bound the term $I_{2}$ in \eqref{eqn.I12}, let us set  $\varphi(u,u')= (R_{t_{k}u}^{t_{k'}u'})^{2}$. We
will first estimate the $\rho$-var norm of $\varphi$. To this aim, we decompose the rectangular increments  of $\varphi$ as follows 
\begin{align}\label{eqn.vp}
  \varphi_{vu}^{v'u'}  
&=   R_{vu}^{v'u'} ( R_{t_{k}u}^{t_{k'}v'} +R_{t_{k}v}^{t_{k'}v'}  )+ 
R_{vu}^{t_{k'}u'} ( R_{t_{k}u}^{v'u'} +R_{t_{k}v}^{v'u'}  ) .
\end{align}
From this decomposition it is readily checked using Hypothesis \ref{hyp.w} that for $ [v,u]\times[v',u']\subset D_{kk'}$  we have 
\begin{multline}\label{eqn.52}
| \varphi_{vu}^{v'u'} |
\leq 
C
 \omega_{R}([v,u]\times[v',u'])^{1/\rho}
  \cdot \omega_{R}(D_{kk'})^{1/\rho} \\
 +
 C
 \omega_{R}([v,u] \times [t_{k'}, t_{k'+1}])^{1/\rho}\cdot
  \omega_{R}([t_{k}, t_{k+1}]\times[v',u'])^{1/\rho}
  .
\end{multline}
This inequality can be used in order to evaluate the 2-d $\rho$-var norm of $\varphi$ over the rectangle~$D_{kk'}$. Indeed,  
let $\cp$ and $\cp'$ be   partitions of $[t_{k}, t_{k+1}]$ and $[t_{k'}, t_{k'+1}]$, respectively.  
From~\eqref{eqn.52} we have
\begin{multline*}
\sum_{(v,u)\in \cp, (v',u')\in \cp'}
| \varphi_{vu}^{v'u'} |^{\rho}
\leq \omega_{R}(D_{kk'}) 
\sum_{(v,u)\in \cp, (v',u')\in \cp'}  
\omega_{R}([v,u]\times[v',u'])
 \\
+
 \sum_{(v,u)\in \cp} 
\omega_{R}( [v,u]
\times [t_{k'}, t_{k'+1}] )
\cdot
\sum_{ (v',u')\in \cp'} 
\omega_{R}([t_{k}, t_{k+1}]\times[v',u'])
 \, ,
\end{multline*}
which, by the super-additivity of $\omega_{R}$, easily yields
\begin{equation*}
\sum_{(v,u)\in \cp, (v',u')\in \cp'}
| \varphi_{vu}^{v'u'} |^{\rho}
\leq 
\omega_{R}(D_{kk'})^{2} . 
\end{equation*}
Since $\cp$ and $\cp'$ are generic partitions of $[t_{k}, t_{k+1}]$, this implies that
\begin{align}\label{eqn.53}
\|\varphi\|_{\rho\text{-var}; D_{kk'}} \leq  \omega_{R}(D_{kk'})^{2/\rho} .
\end{align}
With relation \eqref{eqn.53} in hand, we can now establish a bound for $I_{2}$. Indeed, recall that $\varphi (u,u') = (R_{t_{k}u}^{t_{k'}u'})^{2}$. In particular, we have $\varphi (t_{k}, u')=0$ and $\varphi(u,t_{k'})=0$, and we get
\begin{align*}
\int_{D_{kk'}} (R_{t_{k}u}^{t_{k'}u'})^{2}dR(u,u') = \int_{D_{kk'}} \varphi_{t_{k}u}^{t_{k'}u'} dR(u,u'). 
\end{align*}
Therefore, 
plugging \eqref{eqn.53} in the definition \eqref{eqn.I12} of $I_{2}$ and  applying Lemma \ref{lemma.YLT},   we end up with
\begin{align}\label{eqn.54ii}
|I_{2} | \leq C\cdot \sum_{k,k'} \omega_{R}(D_{kk'})^{3/\rho} .
\end{align}

The estimation of $I_{1}$ can be done along the same lines as for $I_{2}$. Namely, we define a function   $\psi(u,u') = R_{t_{k}u}^{t_{k}u}R_{t_{k'}u'}^{t_{k'}u'}$. Then the rectangular increments of $\psi$ can be decomposed as  
\begin{align*}
\psi_{vu}^{v'u'}  
&
 = 
 (R_{t_{k}u}^{vu} + R_{t_{k}v}^{vu} )( R_{t_{k'}u'}^{v'u'}+  R_{t_{k'}v'}^{v'u'}  ).
\end{align*}
from which we can deduce 	\begin{align}\label{eqn.psi.rec}
|\psi_{vu}^{v'u'}  |  
 &\leq C \cdot
 \omega_{R}([t_{k}, t_{k+1}]\times [v,u])^{1/\rho}\cdot
 \omega_{R}([t_{k'}, t_{k'+1}]\times [v',u'])^{1/\rho}
   . 
\end{align}
Starting from \eqref{eqn.psi.rec} we can proceed  as in \eqref{eqn.53}-\eqref{eqn.54ii}. We arrive at  
\begin{align*}
\|\psi\|_{\rho\text{-var}; D_{kk'}}  \leq C\omega_{R}(D_{kk'})^{2/\rho} .
\end{align*}
Now applying Lemma \ref{lemma.YLT} again we end up with
	\begin{align}\label{eqn.54i}
|I_{1}|&\leq C\sum_{k,k'}  \omega_{R}(D_{kk'})^{3/\rho} 
 .
\end{align}
	Let us conclude our estimates for this step: plugging \eqref{eqn.54i} and \eqref{eqn.54ii} into \eqref{eq:g-iil-bound-initial} 
	we have obtained
	\begin{align}\label{eqn.giil}
\|\delta g^{ii\ell}_{st}\|_{2}^{2} \leq C \sum_{k,k'} \omega_{R} (D_{kk'})^{3/\rho} .
\end{align}

	 \noindent   \textbf{Step 4: Conclusion.} 	 
Let us summarize our considerations so far. 
Gathering the upper bounds  \eqref{eq:iii-bound}, \eqref{eqn.51} and \eqref{eqn.giil}   we have proved relation  \eqref{eq:g-bound} for all possible values of $i,j,\ell \in \{1,\dots, d\}$ and $p=2$.   Recall again that the general case $p\geq 2$ is obtained by hypercontractivity, which finishes our proof.
\end{proof}

Let $g$ be the increment in relation \eqref{eqn.g}. We now wish to obtain an upper bound for $g$ similar to the bound \eqref{eq:F-bound} we have derived for $F$. This is the content of the next proposition. 
 
\begin{proposition}\label{prop.g}
Let $X$ and $g$ be as in Lemma \ref{lemma:g-bound}. Let $\theta>1$ be such that $\frac{1}{\theta}+\frac{1}{\rho}=1$.
  Then for all $i,j,\ell \in \{1,\dots, d\}$ and $(s,t)\in \cs_{2}$ we have 
\begin{multline}\label{eqn.57}
\|\delta g_{st}^{ij\ell}\|_{p }^{2} \\
\leq
 C
   \max_{k,k'}
    \omega_{R}(D_{kk'})^{3/\rho-1} 
       \cdot
       \omega_{R}([s,t]^{2}) 
        + C  \max_{k} \omega_{R}(D_{kk})^{2(1/\rho-1/\theta)}  
\cdot
\omega_{R} ( [s,t]^{2})^{2 /\theta+1/\rho}  .
\end{multline}
In particular, for $\ep $ such that  $ 0\leq\ep \leq (3-\rho)\wedge (2-2\rho/\theta)$ we have
\begin{align}\label{eqn.58}
\|\delta g_{st}^{ij\ell}\|_{p}^{2} \leq   \max_{k,k'} 
\omega_{R}(D_{kk'})^{\frac{\ep}{\rho}}
   \cdot 
\omega_{R}([s,t]^{2})^{\frac{3-\ep}{\rho}} .
\end{align}

\end{proposition}
\begin{proof}
We first observe that since $\rho<2$ and $\frac{1}{\theta}+\frac{1}{\rho}=1$, we have $\theta>2>\rho$. Applying H\"older's inequality to the second term on the right side of \eqref{eq:g-bound} yields
\begin{align*}
 \Big| 
   \sum_{k,k'}   ( R_{t_{k}t_{k+1}}^{t_{k}t_{k+1}}R_{t_{k'}t_{k'+1}}^{t_{k'}t_{k'+1}}R_{t_{k}t_{k+1}}^{t_{k'}t_{k'+1}}) \Big|
\leq\Big( \sum_{k,k'}  | R_{t_{k}t_{k+1}}^{t_{k}t_{k+1}}R_{t_{k'}t_{k'+1}}^{t_{k'}t_{k'+1}} |^{\theta}\Big)^{1/\theta} 
\cdot
\Big( \sum_{k,k'} |R_{t_{k}t_{k+1}}^{t_{k'}t_{k'+1}}|^{\rho} \Big)^{1/\rho}
 .
\end{align*}
Hence thanks to an elementary algebraic manipulation and according to the definition of $\rho$-variation   we get
\begin{eqnarray*}
 \Big| 
   \sum_{k,k'}   
   ( R_{t_{k}t_{k+1}}^{t_{k}t_{k+1}}R_{t_{k'}t_{k'+1}}^{t_{k'}t_{k'+1}}R_{t_{k}t_{k+1}}^{t_{k'}t_{k'+1}}) 
 \Big|
&\leq&
\lp \sum_{k}  | R_{t_{k}t_{k+1}}^{t_{k}t_{k+1}} |^{\te} \rp^{2/\te}
\omega_{R}([s,t]^{2})^{1/\rho} \\
&=&
\lp \sum_{k}  | R_{t_{k}t_{k+1}}^{t_{k}t_{k+1}} |^{\te-\rho} | R_{t_{k}t_{k+1}}^{t_{k}t_{k+1}} |^{\rho} \rp^{2/\te}
\omega_{R}([s,t]^{2})^{1/\rho}.
\end{eqnarray*}
Thus bounding the term $| R_{t_{k}t_{k+1}}^{t_{k}t_{k+1}} |^{\te-\rho}$ above by $\om_{R}(D_{kk})^{(\te-\rho)/\rho}$ and owing to the super additive property of Hypothesis \ref{hyp.w}, we end up with 
\begin{align}\label{eqn.3r}
 \Big| 
   \sum_{k,k'}   ( R_{t_{k}t_{k+1}}^{t_{k}t_{k+1}}R_{t_{k'}t_{k'+1}}^{t_{k'}t_{k'+1}}R_{t_{k}t_{k+1}}^{t_{k'}t_{k'+1}}) \Big|
&\leq \lp\max_{k} 
 \omega_{R}(D_{kk})^{2(1/\rho-1/\theta)}\rp 
\cdot
 \omega_{R}([s,t]^{2})^{2 /\theta} 
\cdot
 \omega_{R}([s,t]^{2})^{1/\rho}
 \nonumber
  \\
&= \lp\max_{k} \omega_{R}(D_{kk})^{2(1/\rho-1/\theta)}\rp 
\cdot
\omega_{R} ( [s,t]^{2})^{2 /\theta+1/\rho}. 
\end{align}
On the other hand, it is easy to see that   
\begin{align}\label{eqn.w3bd}
  \sum_{k,k'}  
  \omega_{R}(D_{kk'})^{3/\rho}
    \leq   \max_{k,k'}
    \omega_{R}(D_{kk'})^{3/\rho-1} 
       \cdot
       \omega_{R}([s,t]^{2}) 
  .
\end{align}
Gathering \eqref{eqn.3r} and \eqref{eqn.w3bd} in \eqref{eq:g-bound}, this   concludes   relation \eqref{eqn.57}. Relation \eqref{eqn.58} follows immediately from   \eqref{eqn.57} and  the fact that $  \omega_{R}(D_{kk'}) \leq \omega_{R}([s,t]^{2})
 $. 
\end{proof}

In the following, we turn to  the estimate of   another third-chaos functional. 
\begin{lemma}\label{lemma.xf}
Let $X$ and $F$ be as in Lemma \ref{lemma:F-bound}. For $i,j,\ell=1,\dots, d$ and $(s,t)\in \cs_{2}$, we define the increment  $h_{st}^{ij\ell}$ as:  
\begin{align*}
h_{st}^{ij\ell} =\sum_{s\leq t_k<t} X_{st_k}^{1,\ell}\delta F_{t_kt_{k+1}}^{ij}. 
\end{align*}
In addition, consider  $\ep $ such that  $0\leq \ep\leq 2-\rho$. Then  the following inequality holds true: 
\begin{align}\label{eqn.cr}
\|h^{ij\ell}_{st}\|_{p}^{2} \leq    4 \lp
 \omega_{R}([s,t]^{2})^{\frac{3}{\rho}-\frac{2\ep}{\rho}}
  + 
  \omega_{R}([s,t]^{2})^{\frac{3}{\rho}-\frac{\ep}{\rho}}
  \rp
 \cdot \max_{k }
\omega_{R}([t_{k}, t_{k+1}] \times [0,T])^{\frac{2\ep}{\rho}}  
.
\end{align}

\end{lemma}
\begin{proof}	
As in the proof of Lemma \ref{lemma:g-bound} we should distinguish cases according to possible equalities in the indices $i,j,\ell$. We   focus on the case $i=j$ in Step 1 to 3 below, and then  deal with the case $i\neq j$ in Step 4.  

 \noindent   \textbf{Step 1: A decomposition of $\mathbf{\|h\|_{2}^{2}}$.} 
 As mentioned above, let us first consider the case $i=j$ and find an estimate for $h^{ii\ell}_{st}$. 
 We start by  writing  
\begin{align}\label{eqn.hex}
	E\lc|h_{st}^{ii\ell}|^2\rc&=\sum_{k,k'}E\left[X_{st_k}^{1,\ell}X_{st_{k'}}^{1,\ell}\delta F_{t_kt_{k+1}}^{ii}\delta F_{t_{k'}t_{k'+1}}^{ii}\right].
\end{align}
In addition, recall from \eqref{eqn.dx} and \eqref{eq:hermite} that 
\begin{align}\label{eqn.dsdf}
\delta X^{i}_{st} = \delta^{\diamond} \lp \mathbf{1}_{[s,t]} e_{i} \rp \, ,
\quad\text{and}\quad 
\delta F^{ii}_{t_{k}t_{k+1}} = \frac12 \si^{2}_{t_{k}t_{k+1}} H_{2} \lp \frac{X^{1,i}_{t_{k}t_{k+1}}}{\si_{t_{k}t_{k+1}}} \rp.
\end{align}
Therefore, one can recast \eqref{eqn.hex} as 
\begin{align}\label{eqn.hl2}
E\lc |h^{ii\ell}_{st}|^{2} \rc = 
\sum_{k,k'} \si^{2}_{t_{k}t_{k+1}} \si^{2}_{t_{k'}t_{k'+1}} 
E \lc \delta^{\diamond} \lp \mathbf{1}_{[s,t_{k}]} e_{\ell} \rp Z_{kk'} \rc,
\end{align}
where the random variable $Z_{kk'}$ is defined by 
\begin{align}\label{eqn.z}
Z_{kk'} = X^{1,\ell}_{st_{k}} H_{2} \lp \frac{X^{1,i}_{t_{k}t_{k+1}}}{\si_{t_{k}t_{k+1}}} \rp
H_{2} \lp \frac{X^{1,i}_{t_{k'}t_{k'+1}}}{\si_{t_{k'}t_{k'+1}}} \rp. 
\end{align}
Hence resorting to the integration by parts formula \eqref{eqn.ibp}, we get that: 
\begin{align*}
E\lc |h^{ii\ell}_{st}|^{2} \rc = \sum_{k,k'} \si_{t_{k}t_{k+1}}^{2} \si_{t_{k'}t_{k'+1}}^{2} E \lc \langle \mathbf{1}_{[s, t_{k}]} e_{\ell}, \mathbf{D} Z_{kk'} \rangle_{\ch} \rc,
\end{align*}
where we recall that $e_{\ell}$ stands for the $\ell$-th element of the canonical basis in $\R^{d}$.
Computing the Malliavin derivative of $Z_{kk'}$ (and recalling that $H_{2}'(x)=x$), we let the reader check that we get the formula 
\begin{align}\label{eqn.3j}
E\lc|h_{st}^{ii\ell}|^2\rc =    \sum_{k,k'} \lp J_{kk'}^{1} +J_{kk'}^{2}+J_{kk'}^{3} \rp,
\end{align}
where the terms $J_{kk'}^{1} $, $J_{kk'}^{2}$, $J_{kk'}^{3}$ are respectively defined by 
\begin{align}\label{eqn.j123}
J_{kk'}^{1}& = 
 E  \left[X_{st_k}^{1,\ell}X_{st_{k'}}^{1,\ell}\right] \cdot E\lc\delta F_{t_kt_{k+1}}^{ii}\delta F_{t_{k'}t_{k'+1}}^{ii}\rc
\nonumber
\\
J_{kk'}^{2}&=   \langle \mathbf{1}_{[s,t_{k}]},  \mathbf{1}_{[t_{k},t_{k+1}]}   \rangle_{\ch} \cdot\langle \mathbf{1}_{[s,t_{k'}]},  \mathbf{1}_{[t_{k'},t_{k'+1}]}   \rangle_{\ch} \cdot
\langle  \mathbf{1}_{[t_{k},t_{k+1}]},  \mathbf{1}_{[t_{k'},t_{k'+1}]}   \rangle_{\ch} \cdot \mathbf{1}_{\{i=\ell\}}
\\
J_{kk'}^{3}&= \langle \mathbf{1}_{[s,t_{k}]},  \mathbf{1}_{[t_{k'},t_{k'+1}]}   \rangle_{\ch} \cdot\langle \mathbf{1}_{[s,t_{k }]},  \mathbf{1}_{[t_{k},t_{k +1}]}   \rangle_{\ch} \cdot
\langle  \mathbf{1}_{[t_{k},t_{k+1}]},  \mathbf{1}_{[t_{k'},t_{k'+1}]}   \rangle_{\ch} \cdot  \mathbf{1}_{\{i=\ell\}}.
\nonumber
\end{align}
In the following, we show that  the upper-bound in \eqref{eqn.cr} holds  for each   $J_{kk'}^{a}$, $a=1,2,3$, and therefore concludes the lemma. 

 \noindent   \textbf{Step 2: Estimate for $\mathbf{J_{kk'}^{1}}$.} 
In order to bound $J^{1}_{kk'}$, we use the definition \eqref{eqn.rcov} as well as Hypothesis \ref{hyp.w} in order to get   
\begin{align*}
\lln E [X^{1,\ell}_{st_{k}} X^{1,\ell}_{st_{k'}}  ] \rrn 
= \lln R^{st_{k'}}_{st_{k}}\rrn 
\leq \omega_{R} \lp [s,t_{k}]\times [s,t_{k'}]\rp^{1/\rho} 
\leq \omega_{R}\lp [s,t]^{2}\rp^{1/\rho}.  
\end{align*}
Furthermore, invoking relations \eqref{eq:hermite} and \eqref{eq:ii-bound}, we get
\begin{align*}
\lln  E [ \delta F^{ii}_{t_{k}t_{k+1}}  \delta F^{ii}_{t_{k'}t_{k'+1}}  ] \rrn 
= \frac18 \lp R^{t_{k}t_{k+1}}_{t_{k'}t_{k'+1}}\rp^{2}
\leq   \omega_{R} \lp D_{kk'}\rp^{2/\rho}.  
\end{align*}
Hence resorting to the same arguments 
  as in  Lemma \ref{lemma:F-bound} 
  in order to get inequality \eqref{eq:ii-bound-2}, 
  we get that   for any $\ep\leq 2-\rho$ we have: 
\begin{align}\label{eqn.j1bd}
\sum_{k,k'} |J_{kk'}^{1}| \leq  \omega_{R}([s,t]^{2})^{1/\rho} \cdot  \sum_{k,k'} 
\omega_{R}(D_{kk'})^{2/\rho} \leq 
  \max_{k} \omega_{R}(D_{kk'})^{\frac{\ep}{\rho}} \cdot
\omega_{R} ( [s,t]^{2})^{\frac{3-\ep}{\rho}} . 
\end{align}
Otherwise stated, inequality \eqref{eqn.cr} is satisfied for $\sum_{k,k'}|J_{kk'}^{1}|$. 
 
 \noindent   \textbf{Step 3: Estimate for $\mathbf{J_{kk'}^{2}}$ and $\mathbf{J_{kk'}^{3}}$.} 
We turn to an upper bound of the term $J_{kk'}^{2}$ in~\eqref{eqn.j123}.  
To this aim, 
consider $\theta>2$   such that $\frac{1}{\theta}+\frac{1}{\rho}=1$. Then   applying H\"older's inequality to  the summation in $k,k'$ of \eqref{eqn.j123},   we get
\begin{multline*}
\sum_{k,k'} 
J_{kk'}^{2} 
\leq \lp\sum_{k,k'} 
\Big|  \langle \mathbf{1}_{[s,t_{k}]},  \mathbf{1}_{[t_{k},t_{k+1}]}   \rangle_{\ch} \langle \mathbf{1}_{[s,t_{k'}]},  \mathbf{1}_{[t_{k'},t_{k'+1}]}   \rangle_{\ch} 
\Big|^{\theta}\rp^{1/\theta} \\
\times
\lp\sum_{k,k'} \Big|
\langle  \mathbf{1}_{[t_{k},t_{k+1}]},  \mathbf{1}_{[t_{k'},t_{k'+1}]}   \rangle_{\ch}  \Big|^{\rho}\rp^{1/\rho}.
\end{multline*}
Now we apply the estimate $\langle \mathbf{1}_{[u,v]},  \mathbf{1}_{[u',v']} \rangle_{\ch}\leq \|R\|_{\rho\text{-var}, [u,v]\times [u',v']}$ to the three inner products   in the above inequality. We arrive at    
\begin{align}\label{eqn.j2bd2}
\sum_{k,k'} 
J_{kk'}^{2} 
&\leq  \lp\sum_{k,k'} 
\Big|  \|R\|_{\rho\text{-var}, [t_{k}, t_{k+1}]\times [s,t]} \cdot
 \|R\|_{\rho\text{-var}, [t_{k'}, t_{k'+1}]\times [s,t]} 
\Big|^{\theta} \rp^{1/\theta}
\cdot \lp\sum_{k,k'} \|R\|_{\rho\text{-var}, D_{kk'}}^{\rho}\rp^{1/\rho}
\nonumber 
\\
&\leq 
  \lp\sum_{ k } 
\|R\|_{\rho\text{-var}, [t_{k }, t_{k +1}]\times [s,t]}^{\theta}  \rp^{2/\theta} 
\cdot
\omega_{R}([s,t]^{2})^{1/\rho} 
. 
\end{align}
Similarly to what we have done in the proof of Proposition \ref{prop.g}, we note that   $\theta>2>\rho$. 
Hence combining   Hypothesis \ref{hyp.w} and the super-additivity of $\omega_{R}$ we get  the following estimate for   $0<\ep\leq \theta-\rho$:
\begin{align}\label{eqn.j2bd}
\sum_{k,k'}J_{kk'}^{2} \leq 
\max_{k }
\omega_{R}([t_{k}, t_{k+1}] \times [s,t])^{\ep/\rho} 
\cdot
\max_{k'}
\omega_{R}([t_{k'}, t_{k'+1}] \times [s,t])^{\ep/\rho} 
 \cdot   
\omega_{R}([s,t]^{2})^{\frac{3-2\ep}{\rho}}
. 
\end{align}
This implies that the upper-bound  estimate  \eqref{eqn.cr} holds for $J_{kk'}^{2}$. 
It can also be shown that the same estimate holds for   $J_{kk'}^{3}$. The proof is   similar to that of $J_{kk'}^{2}$ and will be omitted.  
Combining \eqref{eqn.j1bd} and \eqref{eqn.j2bd},    
  this     completes the proof of \eqref{eqn.cr} for $i=j$. 
  
   \noindent   \textbf{Step 4: The case $i\neq j$.}   
   We now turn to an estimate of $h^{ij\ell}_{st}$ when $i\neq j$. 
  To this aim we first write an expression for $E[ |h^{ij\ell}_{st}|^{2} ]$ mimicking   \eqref{eqn.hex}, with the important difference that the term $\delta F^{ij}_{t_{k}t_{k+1}}$ cannot be represented by Hermite polynomials as in \eqref{eqn.dsdf}. Hence the equivalents  for~\eqref{eqn.hl2} and \eqref{eqn.z}  whenever $i\neq j$ is 
  \begin{align}\label{eqn.hl2d}
E\lc |h^{ij\ell}_{st}|^{2} \rc= \sum_{k,k'}   \lc \delta^{\diamond} ( \mathbf{1}_{[s,t_{k}]} e_{\ell} ) \tilde{Z}_{kk'}  \rc,
\end{align}
where 
\begin{align*}
\tilde{Z}_{kk'} = X^{1,\ell}_{st_{k'}} 
\cdot
\delta^{\diamond} ( X^{1,i}_{t_{k}\cdot} 
\cdot
\mathbf{1}_{[t_{k}, t_{k+1}]} e_{j} ) \cdot \delta^{\diamond}  ( X^{1,i}_{t_{k'}\cdot} \cdot
\mathbf{1}_{[t_{k'}, t_{k'+1}]} e_{j} ) . 
\end{align*}
Integrating relation \eqref{eqn.hl2d} by parts similarly to \eqref{eqn.3j}, we end up with 
 \begin{align*}
E\lc |h^{ij\ell}_{st}|^{2} \rc = \sum_{k,k'} (J_{kk'}^{1}+J_{kk'}^{4}),
\end{align*}
where $J^{1}_{kk'}$ has already been defined in \eqref{eqn.j123} and $J^{4}_{kk'}$ is given by 
\begin{align}\label{eqn.j4}
J_{kk'}^{4}= &
\int_{ D_{kk'}}   \lp
 R_{st_{k} }^{t_{k'}u} R_{st_{k'} }^{t_{k }u' }  + R_{st_{k} }^{t_{k }u' } R_{st_{k'} }^{t_{k'}u}  \rp
 dR(u,u') \cdot \mathbf{1}_{\{i=\ell\}} 
 \nonumber
 \\
&+
\int_{ D_{kk'}}   \lp
 R_{st_{k} }^{ut_{k'+1}} R_{st_{k'} }^{u't_{k +1} }  + R_{st_{k} }^{u't_{k +1} } R_{st_{k'} }^{ut_{k'+1}}  \rp
 dR(u,u') \cdot \mathbf{1}_{\{j=\ell\}} 
.
\end{align}
In order to bound $J^{4}_{kk'}$, we set 
  $\phi (u,u') =
R_{st_{k} }^{t_{k'}u} R_{st_{k'} }^{t_{k }u' } $. Then one of the terms in \eqref{eqn.j4} is  $\int_{D_{kk'}} \phi(u,u')dR(u,u')$. 
We wish to bound this term thanks to Lemma \ref{lemma.YLT}. To this aim, similarly to what we did in \eqref{eqn.vp}-\eqref{eqn.52}, we estimate the rectangular increments of $\phi$. We get   
\begin{align*}
|\phi_{uv}^{u'v'}| = 
|R_{st_{k} }^{uv} R_{st_{k'} }^{u'v' }| \leq 
\|R\|_{\rho\text{-var}, [s,t]\times [u,v]}
 \cdot
 \|R\|_{\rho\text{-var}, [s,t]\times [u',v']}
      . 
\end{align*}
Now we consider 
  $\cp$ and $\cp'$ generic  partitions of $[t_{k}, t_{k+1}]$ and $[t_{k'}, t_{k'+1}]$  respectively, as well as $\theta$ such that $\frac{1}{\theta}+\frac{1}{\rho}=1$. Then we have
\begin{align*}
\sum_{[u,v]\times [u',v']\in \cp\times \cp'} |\phi_{uv}^{u'v'}|^{\theta} \leq \sum_{[u,v]\in \cp} \|R\|_{\rho\text{-var}, [s,t]\times [u,v]}^{\theta }
\cdot
 \sum_{[u',v']\in \cp'} \|R\|_{\rho\text{-var}, [s,t]\times [u',v']}^{\theta}
 \\
\leq
\|R\|_{\rho\text{-var}, [s,t]\times [t_{k}, t_{k+1}]}^{\theta}
 \cdot
 \|R\|_{\rho\text{-var}, [s,t]\times [t_{k'}, t_{k'+1}]}^{\theta} 
  .
\end{align*}
Therefore, we obtain
\begin{align}
\label{eqn.phi}
\|\phi\|_{\theta\text{-var}, D_{kk'}}\leq
 \|R\|_{\rho\text{-var}, [s,t]\times [t_{k}, t_{k+1}]} 
 \cdot
 \|R\|_{\rho\text{-var}, [s,t]\times [t_{k'}, t_{k'+1}]}.
\end{align}
Note that the estimate \eqref{eqn.phi} of $\phi$ also holds for the other three functions in the right-hand side of \eqref{eqn.j4}, namely:
\begin{align*}
R_{st_{k} }^{t_{k }u' } R_{st_{k'} }^{t_{k'}u} , \qquad R_{st_{k} }^{ut_{k'+1}} R_{st_{k'} }^{u't_{k +1} }  , \qquad R_{st_{k} }^{u't_{k +1} } R_{st_{k'} }^{ut_{k'+1}}.
\end{align*}
The proof is similar and will be left to the reader. 
With \eqref{eqn.phi} in hand, we can now invoke
  Lemma \ref{lemma.YLT} for the right-hand side of relation \eqref{eqn.j4}. This yields
\begin{align}\label{eqn.j4bd}
|J_{kk'}^{4}| \leq  \|R\|_{\rho\text{-var}, [s,t]\times [t_{k}, t_{k+1}]} 
 \cdot
 \|R\|_{\rho\text{-var}, [s,t]\times [t_{k'}, t_{k'+1}]} \cdot \|R\|_{\rho\text{-var}, D_{kk'}}. 
\end{align}
Starting from \eqref{eqn.j4bd}, we easily get an upper bound similar to  \eqref{eqn.j2bd2} for $\sum_{k,k'}J^{4}_{kk'}$. 
 Then we can proceed as in relation  \eqref{eqn.j2bd}. 
 We conclude that \eqref{eqn.cr} holds for the case $i\neq j$. The proof is now complete. 
\end{proof}

\subsection{Upper-bounds for weighted sums}\label{sec:bounds-weighted-sums}
In this section we give some estimates for weighted sums of the processes $F$ and $g$ defined in the previous subsection. These sums will be a part of our main terms in the analysis of the trapezoid rule.  
\begin{lemma}\label{lem:weighted-F}
	Let $X$ be a $\R^d$ valued   Gaussian process with covariance function $R$ such that Hypothesis \ref{hyp.x} holds with $\rho\in [1,2)$, and therefore Hypothesis~\ref{hyp.w} is guaranteed by Remark~\ref{remark.rho} and subsequent comments. Let $\omega_{R}$ be the control  in  Hypothesis~\ref{hyp.w}. 
	Recall that the increment 
	    $F$ is   defined in \eqref{F}   and fix a partition $\cp$ with mesh $|\cp|$.  	     
We also consider a controlled process of order $1$ according to Definition \ref{def:ctrld-process}, which means that the increments of $y$ can be decomposed as   \begin{align}\label{eqn.y}
y_{st} = y^{1}_{s}X^{1}_{st} + r_{st} \qquad s,t\in [0,T] . 
\end{align}
We call $\omega$ the control  $\omega_{y}$ related to the increments of $y$ in Definition \ref{def:ctrld-process},  and recall that we have  
\begin{align}\label{eqn.y1r}
 |\delta y^{1}_{st}|\leq  \omega(s,t)^{1/p} , 
 \qquad
  |r_{st}|\leq  \omega(s,t)^{2/p},  \qquad \text{for all}\quad (s,t)\in \cs_{2}([0,T]) 
\end{align}
 almost surely. Eventually, we introduce  below a parameter $p$ such that $\frac{1}{p} = \frac{1-\ep}{2\rho}$ for $\ep$ small enough. 
 Then the following holds true:
 
 \noindent\emph{(i)}
  For every $M>0$,  we set $A_{M}=\{\omega(0,T)\leq M\}$. Then for all   $(s,t)\in \cs_2 $ and $(i,j)\in \{1,\ldots,d\}^{2}$ we have:
	\begin{align}\label{eq:weighted-F}
	 E\lc \mathbf{1}_{A_{M}}\cdot	\left|\sum_{s\leq t_k<t} y_{t_k} \delta F_{t_k t_{k+1}}^{ij}\right|\rc
	  \leq 
C  \cdot \max_{k } 
\omega_{R}([t_{k}, t_{k+1}] \times [0,T])^{\frac{\ep}{2\rho}}    
.
	\end{align}
In particular,  \begin{align*}
\sum_{s\leq t_k<t} y_{t_k} \delta F_{t_k t_{k+1}}^{ij} \longrightarrow 0, 
\qquad \text{in probability as $|\cp|\to0$}. 
\end{align*}

\noindent\emph{(ii)} 
	In case of H\"older continuous processes $X$ and $y$, one can improve the convergence as follows. Namely 
	 suppose   that $\delta y^{1}\in \mathcal{C}^{1/p}$ and $r\in \cac^{2/p}$ almost surely and that 
	 \begin{align}\label{eqn.Rholder}
 \omega_{R}([s,t]\times [0,T]) \leq C|t-s|,
   \qquad \text{for all } (s,t)\in \cs_{2}([0,T]) .
\end{align} 
 Furthermore, assume 
   that the uniform partition $0=t_{0}<\cdots<t_{n}=T$ over $[0,T]$ is considered.   Then
	 \begin{align}\label{eqn.xfconv}
\sum_{s\leq t_k<t} y_{t_k} \delta F_{t_k t_{k+1}}^{ij} \longrightarrow 0, 
\qquad \text{almost surely as $n\to\infty$}. 
\end{align}

\end{lemma}
\begin{proof}
\noindent   \textbf{Step 1: A decomposition.} 
 Consider $(s,t)\in \cs_2  $, where we recall that $\cs_{2}$ stands for $\cs_{2}([0,T])$.  With the help of \eqref{eqn.y} we have the following decomposition
	\begin{align}\label{eqn.yf}
 \sum_{s\leq t_{k}<t} y_{st_{k}} \delta F^{ij}_{t_{k}t_{k+1}}  = 
 y^{1}_{s}\sum_{s\leq t_{k}<t} X^{1}_{st_{k}} \delta F^{ij}_{t_{k}t_{k+1}}  + \tilde{r}_{st}  , 
\end{align}
where we denote 
\begin{align}\label{eqn.tr}
\tilde{r}_{st}  =   \sum_{s\leq t_{k}<t} r _{st_{k}} \delta F^{ij}_{t_{k}t_{k+1}}. 
\end{align}

\noindent   \textbf{Step 2: Calculations for $\delta \tilde{r}$.} 
In the following, in order to estimate $\tilde{r}$  we first estimate $\delta \tilde{r}$. To this aim we observe that a simple computation yields  $\delta r_{sut} = \delta y^{1}_{su} X^{1}_{ut}$.  
Therefore, starting from~\eqref{eqn.tr} and using Definition \ref{def:delta-on-C1-C2} for the increment $\delta \tilde{r}$, some elementary calculations show that  
\begin{align}\label{eqn.dr}
\delta \tilde{r}_{sut} &=  r_{su}\sum_{u\leq t_{k}<t} \delta F^{ij}_{t_{k}t_{k+1}} +  \sum_{u\leq t_{k}<t}\delta r_{sut_{k}} \delta F^{ij}_{t_{k}t_{k+1}} 
\nonumber
\\
&= r_{su}\sum_{u\leq t_{k}<t} \delta F^{ij}_{t_{k}t_{k+1}} + \delta y^{1}_{su} \sum_{u\leq t_{k}<t} X^{1}_{ut_{k}}\delta F^{ij}_{t_{k}t_{k+1}}. 
\end{align}

\noindent   \textbf{Step 3: Moment estimates of $\delta \tilde{r}$ and $\tilde{r}$.} 
We now hinge on relation \eqref{eqn.dr} in order to upper bound $\tilde{r}$. 
We start by denoting 
\begin{align}\label{eqn.trunc}
y^{1,M} = \mathbf{1}_{A_{M}}\cdot y^{1}\,, \quad
\tilde{r}^{M}=\mathbf{1}_{A_{M}}\cdot \tilde{r} \,, 
\quad
  \delta \tilde{r}^{M}=\mathbf{1}_{A_{M}}\cdot \delta \tilde{r}\,, 
\quad \omega^{M} (s,t)=E\big[\mathbf{1}_{A_{M}}\cdot \omega(s,t)\big].
\end{align}
It is easy to  see that $\omega^{M}$ is a control. By the inequalities in \eqref{eqn.y1r} we also have
\begin{align}\label{eqn.ym1rm}
E\lc|y^{1,M}_{su}|^{p}\rc\leq 
  \omega^{M}(s,u) , 
\qquad\text{and}\qquad
E\lc| r^{M}_{su}|^{p/2}\rc\leq  \omega^{M}(s,u) .
\end{align}
We also recall that  Hypothesis \ref{hyp.w} holds with some 2-d control $\omega_{R}$.

Next  we multiply  both sides of  \eqref{eqn.dr} by $\mathbf{1}_{A_{M}}$. According to our notation \eqref{eqn.trunc}, we get 
\begin{align}\label{eqn.drm}
\delta r_{sut}^{M} = r^{M}_{su}\sum_{u\leq t_{k}<t} \delta F^{ij}_{t_{k}t_{k+1}} + \delta y^{1,M}_{su} \sum_{u\leq t_{k}<t} X^{1}_{ut_{k}} \delta F^{ij}_{t_{k}t_{k+1}}. 
\end{align}
We are now in a position to apply H\"older's inequality, Lemma \ref{lemma:F-bound}, Lemma \ref{lemma.xf} and the upper bound \eqref{eqn.ym1rm} in order to get
\begin{align}\label{eqn.debd}
E\lc|\delta \tilde{r}^{M}_{sut}|\rc 
\leq 
&
2\omega^{M}(s,t)^{2/p} \cdot 
\omega_{R} ([s,t]^{2})^{\frac{1}{\rho}-\frac{\ep}{2\rho}}
  \cdot \max_{k,k'} 
\omega_{R}(D_{kk'})^{\frac{\ep}{2\rho}}
 \nonumber
\\
&+ 2 \omega^{M}(s,t)^{1/p} \cdot 
\omega_{R}([s,t]^{2})^{\frac{3}{2\rho}-\frac{\ep}{2\rho}}
  \cdot \max_{k }
\omega_{R}([t_{k}, t_{k+1}] \times [0,T])^{\frac{\ep}{2\rho}}  
 . 
\end{align}
Let us now discuss the exponents in \eqref{eqn.debd}. Indeed, recall that we have chosen $p$ such that $\frac{2}{p} = \frac{1-\ep}{2\rho}$. Therefore, owing to the fact that $\rho\in [1,2)$, $\ep$ can be chosen small enough so that 
 \begin{align}\label{eqn.nu}
\nu_{p,\rho} \equiv \lp \frac{2}{p} + \frac{1-\ep/2}{\rho}\rp 
\wedge 
\lp \frac{1}{p} + \frac{3}{2} -\frac{\ep}{2} \rp>1 . 
\end{align}
In the sequel we pick a 
 $\mu   $ such that $1<\mu<\nu_{p, \rho}$. With this notation in hand define a bivariate function   $\tilde{\omega}$ by  
\begin{align*}
\tilde{\omega} (s,t)  =
\lp
 \omega^{M}(s,t)^{2/p} \cdot \omega_{R} ([s,t]^{2})^{\frac{1}{\rho}-\frac{\ep}{2\rho}} 
 \rp^{1/\mu}
+
\lp
  \omega^{M}(s,t)^{1/p} \cdot 
\omega_{R}([s,t]^{2})^{\frac{3}{2\rho}-\frac{\ep}{2\rho}}
\rp^{1/\mu}.
\end{align*}
As a direct application of \cite[Exercise 1.9 item (iii)]{FV}, it is readily checked that $\tilde{\omega}$ is a control. In addition, one can recast    
 \eqref{eqn.debd}   as 
\begin{align}\label{eqn.drmbd}
E(|\delta \tilde{r}^{M}_{sut}|) 
\leq 2  \max_{k }
\omega_{R}([t_{k}, t_{k+1}] \times [0,T])^{\frac{\ep}{2\rho}}    
  \cdot \tilde{\omega}(s,t)^{\mu}. 
\end{align}

Summarizing our considerations for this step, we have obtained that $\tilde{r}$ is an increment from $\cs_{2}$ to the Banach space $\cb=L^{1}(\Omega)$. Moreover, relation \eqref{eqn.tr} easily entails   $\tilde{r}^{M}_{t_{\ell}t_{\ell+1}}=0$ for any point $t_{\ell}$ of the partition $\cp$, and $\tilde{r}$ satisfies \eqref{eqn.drmbd}. Therefore, a direct application of  Lemma \ref{lem2.4} on $\cb=L^{1}(\Omega)$ yields
\begin{align*}
E(|  \tilde{r}_{st}^{M} |) 
\leq 2 K_{\mu} \max_{k }
\omega_{R}([t_{k}, t_{k+1}] \times [0,T])^{\frac{\ep}{2\rho}}    
  \cdot \tilde{\omega}(s,t)^{\mu}. 
\end{align*}
Plugging this estimate into \eqref{eqn.yf} and combining it with 
\eqref{eqn.cr}, the proof of our claim \eqref{eq:weighted-F} is now achieved.

\noindent   \textbf{Step 4: Path-wise estimates of $F$.} 
We now turn to item (ii) in our lemma, assuming H\"older-continuity for 
$y,y^{1}, r$ and considering uniform partitions of $[0,T]$ with $t_{k+1}-t_{k}=T/n$. In this context,   condition \eqref{eqn.Rholder} allows   to write    the upper-bound estimate of $F$ \eqref{eq:F-bound} in Lemma~\ref{lemma:F-bound} as:
\begin{align*}
\|\delta F^{ij}_{st}\|_{q} \leq C  |t-s|^{\frac{1}{\rho}-\ep} \cdot |\cp|^{ \ep  } = C  |t-s|^{\frac{1}{\rho}-\ep} \cdot n^{- \ep  }, \quad \text{for all } q>1  \text{ and }   (s,t)\in \cs_{2}. 
\end{align*}
 Applying Lemma \ref{G existence} with $z^{n}=F$, 
 $\beta=\frac{1}{\rho}-\ep$, and $\al=\ep$ 
 we obtain
 \begin{align}\label{eqn.Fas}
|\delta F^{ij}_{st}| \leq G \cdot  (t-s)^{\frac{1}{\rho}-2\ep} \cdot n^{- \ep /2}, 
\end{align}
where $G$ is a random variable admitting moments of all orders.
In a similar way and with the help of Lemma \ref{lemma.xf}, we can show that
\begin{align}\label{eqn.xfas}
\Big| \sum_{s\leq t_{k}<t} X^{1 }_{ut_{k}} F^{ij}_{t_{k}t_{k+1}} \Big| \leq G\cdot |t-s|^{\frac{3}{\rho} - 2\ep}\cdot n^{-  \ep /4}. 
\end{align}

With those preliminaries in mind, we will upper bound the increment 
$\sum_{s\leq t_{k}<t} y_{t_{k}} \delta F^{ij}_{t_{k}t_{k+1}}$ thanks to relation \eqref{eqn.yf}. Namely in the right-hand side of \eqref{eqn.yf} we have that almost surely 
\begin{align*}
y_{s}\sum_{s\leq t_{k}<t} X^{1}_{st_{k}} \delta F^{ij}_{t_{k}t_{k+1}} \to 0, \quad \text{for all } (s,t)\in \cs_{2}
\end{align*}
thanks to \eqref{eqn.xfas}. Therefore, in order to show     \eqref{eqn.xfconv}  it remains to show the convergence of $\tilde{r}$.

\noindent   \textbf{Step 4: Path-wise estimates of $\delta \tilde{r}$ and $\tilde{r}$.} 
In order to bound $\delta \tilde{r}$ in the H\"older case, we plug 
 \eqref{eqn.Fas} and \eqref{eqn.xfas} into the expression  \eqref{eqn.dr} we have obtained for $\delta \tilde{r}$. We end up with 
\begin{align*}
|\delta \tilde{r}_{sut}| 
&\leq G \lp
\|r\|_{2/p} |t-s|^{\frac{2}{p}+\frac{1}{\rho}-2\ep}\cdot n^{-\ep/2} + \|y^{1}\|_{1/p} |t-s|^{\frac{1}{p}+\frac{3}{\rho}-2\ep} \cdot n^{-\ep/4}
\rp
\\
&\leq n^{-\ep/4} G \lp \|r\|_{2/p}+\|y^{1}\|_{1/p}\rp \cdot |t-s|^{\mu},  
\end{align*}
where similarly to what we did in Step 3, we take $1<\mu<\tilde{\nu}_{p,\rho}$ with 
\begin{align*}
\tilde{\nu}_{p,\rho} =  \lp \frac{2}{p} +\frac{1}{\rho}-2\ep \rp \wedge \lp \frac{1}{p} +\frac{3}{\rho}-2\ep \rp.  
\end{align*}
 Hence one can resort to the H\"older version of the sewing lemma contained in Lemma \ref{prop:sewing}. We get 
\begin{align*}
|\tilde{r}_{st}|\leq  C G (\|r\|_{2/p}+\|y^{1}\|_{1/p}) \cdot n^{-\ep/4} |t-s|^{\mu}, 
\end{align*}
from which we easily deduce 
\begin{align}\label{eqn.rto0}
  \lim_{n\to\infty}|\tilde{r}_{s t}| = 0. 
\end{align}
 In conclusion, plugging \eqref{eqn.rto0} and \eqref{eqn.xfas} into \eqref{eqn.yf} we have obtained relation \eqref{eqn.xfconv}.    
 The proof is   complete.  
\end{proof}

\noindent 
We now handle some weighted sums of the increment $X^{3}$ which will feature in our trapezoid sums. 
\begin{lemma}\label{lem:weighted-g}
As in Lemma \ref{lem:weighted-F}, we consider a Gaussian process $X$ whose covariance $R$ satisfies   Hypothesis \ref{hyp.x} with $\rho\in [1,2)$. We call $\omega_{R}$ the control defined in Hypothesis \ref{hyp.w}. 
Let $X^{3} = \{ X^{3,ij\ell}_{st}; (s,t)\in \cs_{2}([0,T]), i,j,\ell =1,\dots, d \} $ be the third order element in the rough path above $X$ and take a sequence of partitions $\cp$ with mesh $|\cp|$. 
We also consider a continuous process 
      $y$  such that $y_{0}=0$.
      Then the following holds true: 

\noindent
\emph{(i)} 
Let us assume that the increments of $y$ are dominated by a control 
  $\omega $  over $[0,T]$.  Namely we suppose that for all $(s,t)\in \cs_{2}([0,T])$ we have  $|\delta y_{st}| \leq \omega(s,t)^{1/p}$ almost surely, where $p$ is such that $\frac{1}{p}= \frac{1-\ep}{2\rho}$ as in Lemma \ref{lem:weighted-F}.
  Consider a generic index $(i,j,\ell) \in \{ 1,\dots, d \}^{3}$, and recall that $A_{M}$ is defined by $A_{M}= \{ \omega (0,T)\leq M \}$ for $M>0$. Then we have 
\begin{equation}\label{eq:weighted-zeta}
E\lc\mathbf{1}_{A_{M}}	\cdot\left|\sum_{s\leq t_k<t} \delta y_{st_k} X_{t_kt_{k+1}}^{3,ijl}\right| \, \rc
\leq  C\cdot  \max_{k,k'} 
\omega_{R}(D_{kk'})^{\frac{\ep}{2\rho}}
  .
\end{equation}
In particular, 
\begin{align}\label{eqn.yx3conv}
\sum_{s\leq t_k<t} \delta y_{st_k} X_{t_kt_{k+1}}^{3,ijl} 
\longrightarrow 0, \qquad \text{in probability as $|\cp|\to0$.}
\end{align}

\noindent
\emph{(ii)}  
If we are in a H\"older setting, namely $y\in \cac^{1/p}$ and
$\omega_{R}$ verifying \eqref{eqn.Rholder}, and if we also consider the uniform partition $\cp$ (see Lemma \ref{lem:weighted-F} item \emph{(ii)}), then we get  
\begin{align}\label{eqn.yx3as}
\sum_{s\leq t_k<t} \delta y_{st_k} X_{t_kt_{k+1}}^{3,ijl} \longrightarrow 0, \qquad \text{almost surely as $n\to\infty$.}
\end{align}

\end{lemma}
\begin{proof}
The proof is very similar to what we did for Lemma \ref{lem:weighted-F}. For sake of conciseness we will only outline some of the steps, focusing mainly on getting an equivalent of \eqref{eqn.debd}. 
Along the same lines as \eqref{eqn.trunc}, we set  
	 \begin{align*}
	 y^{M}_{s}= \mathbf{1}_{A_{M}} \, y_{s}
\,,	 \qquad 
r_{st}^{M} =
\mathbf{1}_{A_{M}}\cdot  \sum_{s\leq t_k<t} \delta y_{st_k} X_{t_kt_{k+1}}^{3,ijl}.
\end{align*}
In this context we let the reader check that the equivalent of relation \eqref{eqn.drm} becomes    
\begin{align*}
\delta r_{sut}^{M} =  \delta y_{su}^{M}\cdot   \sum_{u\leq t_{k}<t} X^{3,ijl}_{t_{k}t_{k+1}}. 
\end{align*}
Hence applying H\"older's inequality with $p$, $q$ such that $\frac{1}{p}+\frac{1}{q}=1$, invoking Proposition \ref{prop.g} and recalling that $\omega^{M} $ is introduced in relation \eqref{eqn.trunc},   we get  
\begin{align}\label{eqn.drbd}
E\lc|\delta r_{sut}^{M} |\rc 
&
=   E\lc
|\delta y^{M}_{su}|^{p}  \rc^{1/p} \cdot
\Big\| \sum_{u\leq t_{k}<t} X^{3,ijl}_{t_{k}t_{k+1}}\Big\|_{q}
\nonumber
\\
&\leq 
   \omega^{M}(s,u)^{1/p}   
 \cdot   \max_{k,k'} 
\omega_{R}(D_{kk'})^{\frac{\ep}{2\rho}}
   \cdot 
\omega_{R}([s,t]^{2})^{\frac{3-\ep}{2\rho}}
.
\end{align}
Observe that \eqref{eqn.drbd} corresponds to \eqref{eqn.debd} in the proof of Lemma \ref{lem:weighted-F}. 
Also notice that we have chosen $p$ (with a small enough $\ep$) so that 
$\frac{1}{p} + \frac{3-\ep}{2\rho}>1$. Otherwise stated, condition~\eqref{eqn.nu} holds in the current context. Therefore one can prove our claims \eqref{eq:weighted-zeta}, \eqref{eqn.yx3conv}  and \eqref{eqn.yx3as} exactly as in Lemma \ref{lem:weighted-F}. 
\end{proof}

\begin{remark}\label{remark.b}
Combining Lemma \ref{lemma.xf} with the above considerations, one can easily extend the conclusions of Lemma \ref{lem:weighted-g} to sums of the form 
\begin{align*}
\sum_{s\leq t_{k}<t} \delta y_{st_{k}}  X^{2,ij}_{t_{k}t_{k+1}} X^{1,\ell}_{t_{k}t_{k+1}}. 
\end{align*}
Details are ommited for sake of conciseness. 

\end{remark}

\subsection{Convergence of the trapezoid rule}\label{sec:proof-main-thm}
With the previous preliminary results in hand, we are now ready to give  a complete statement and carry out the proof of  our main  Theorem~\ref{thm:cvgce-trapezoid-intro}.
\begin{theorem}\label{theorem:Trapezoid-Rule}
	Let $X$ be a centered $\R^{d}$-valued  Gaussian process on $[0,T]$ with covariance function $R$. Suppose that Hypothesis \ref{hyp.x} holds true for $\rho\in [1,2)$.  Denote the rough path lift of $X$ by $\mathbf{X}=(X^1,X^2,X^3)$.  	
Next consider a $\R^{d}$-valued  controlled process $y$ of order $2$, according to Definition~\ref{def:ctrld-process}. Specifically, there exist processes 	  $  y^{1}, y^{2}, r^{0}, r^{1}$ with regularities to be specified below and such that  
	\begin{align}\label{eqn.yctr}
\delta y_{st} =y_{s}^{1 }X_{st}^{1 }+y_{s}^{2 }X_{st}^{2  }+r_{st}^{0} , 
\qquad \delta y^{1}_{st} = y^{2}_{s}X^{1}_{st}+ r^{1}_{st},
\end{align}  
where we recall again our Notation~\ref{not:tensor-products} on matrix products.
For a given partition of $[0,T]$: $\cp=\{0=t_0<t_1<\cdots<t_{n+1}=T\}$ with mesh size $|\cp|$,  we define the trapezoid rule:
\begin{equation}\label{eq:proof-what-to- prove}
\operatorname{tr-} \cj_0^T(y,X)= \sum_{k=0}^n\frac{y_{t_k} +y_{t_{k+1}} }{2}  \cdot  X_{t_k t_{k+1}}^{1 }, 
\end{equation}
where we denote $y_{t} = ( y_{t}^1 ,\dots, y_{t}^{d} ) $ and  simply write $(y_{t_k} +y_{t_{k+1}})  \cdot  X_{t_k t_{k+1}}^{1 }$ for the inner product $\sum_{i=1}^{d} (y_{t_k}^{i} +y^{i}_{t_{k+1}})  \,  X_{t_k t_{k+1}}^{1, i } $.
Then the following holds true: 

\smallskip

\noindent	\emph{(i)} 
Assume  that the framework of Definition \ref{def:ctrld-process} prevails and call $\omega$ the control function over $[0,T]$ such that for $\frac{1}{p}$ of the form $\frac{1-\ep}{2\rho}$ we have 
 \begin{align}\label{eqn.yctrbd}
|r^{0}_{st}|\leq  \omega(s,t)^{3/p}, \qquad 
|\delta y^{2}_{st}| \leq \omega(s,t)^{1/p},
\qquad 
|r^{1}_{st}|\leq \omega(s,t)^{2/p}.
\end{align}
 Then as the mesh size of the partition $|\cp|$ goes to $0$,   we have 
\begin{align}\label{eqn.trapcp}
\operatorname{tr-} \cj_0^T(y,X)\rightarrow \int_0^T  {y}_s \, d\mathbf{X}_s \qquad \text{in probability},
\end{align}
where the right hand side above designates the rough integral of $y$ against $X$ as given in   Proposition \ref{prop:intg-ctrld-process}.

\noindent\emph{(ii)} 
 Assume we are in a H\"older setting, that is $\omega_{R}$ verifies \eqref{eqn.Rholder}. 
Then recall that  $X$ generates a $\frac{1}{p}$-H\"older rough path, and we also assume that $r^{0}\in \cac^{3/p}$, $r^{1}\in \cac^{2/p}$ and $y,y^{1}, y^{2}\in \cac^{1/p}$.
 Eventually, consider   the uniform partition $0=t_{0}<\cdots<t_{n}=T$ over $[0,T]$ with mesh $T/n$.  Then the convergence in \eqref{eqn.trapcp} holds almost surely. 
\end{theorem}
\begin{proof}
Proceeding with the implicit assumption that we sum over  $t_k\in \cp$, omiting the summation sign for notational convenience we get  
\begin{equation}\label{eq:trapezoid-first}
\operatorname{tr-} \cj_0^T(y,X)= y_{t_{k} }  \cdot X_{t_{k} t_{k+1}}^{1 }
+\frac{1}{2}\delta y_{t_{k} t_{k+1}}  \cdot X_{t_{k} t_{k+1}}^{1 }.
\end{equation}
Hence plugging the decomposition  \eqref{eqn.yctr} into \eqref{eq:trapezoid-first} and invoking Notation~\ref{not:tensor-products} on matrix products with $m=d$, we obtain
\begin{align}\label{eqn.trexp}
\operatorname{tr-} \cj_0^T(y,X) =&
y_{t_{k} }  \cdot X_{t_{k} t_{k+1}}^{1 }
+\frac{1}{2}\Bigg(y_{t_{k} }^{1 }X_{t_{k} t_{k+1}}^{1 }+y_{t_{k} }^{2  }X_{t_{k} t_{k+1}}^{2  }+r_{t_{k} t_{k+1}}^{0 }\Bigg) \cdot X_{t_{k} t_{k+1}}^{1 }. 
\end{align}
Let us rearrange the right-hand side above by  by setting 
\begin{align*}
I_{1} &=    y_{t_{k} }  \cdot X_{t_{k} t_{k+1}}^{1 } + y^{1}_{t_{k}} \cdot X^{2}_{t_{k}t_{k+1}} + y^{2}_{t_{k}} \cdot X^{3}_{t_{k}t_{k+1}}    
\\
 I_{2} &=   \frac12 y^{1}_{t_{k}} X^{1}_{t_{k}t_{k+1}}\cdot X^{1}_{t_{k}t_{k+1}} 
 -  y^{1}_{t_{k}} \cdot X^{2}_{t_{k}t_{k+1}} 
\\
I_{3}&=     \frac12 y^{2}_{t_{k}} X^{2}_{t_{k}t_{k+1}}\cdot X^{1}_{t_{k}t_{k+1}} 
- y^{2}_{t_{k}} \cdot X^{3}_{t_{k}t_{k+1}}   
\\
I_{4}&=\frac12 r^{0}_{t_{k}t_{k+1}} \cdot X^{1}_{t_{k}t_{k+1}},
\end{align*}
where we have written $u\cdot v$ for inner products of vectors as well as matrices.
Then one can recast \eqref{eqn.trexp} as 
\begin{align}\label{eqn.i1234}
\operatorname{tr-} \cj_0^T(y,X)  = I_{1}+I_{2}+I_{3}+I_{4} .
\end{align}
Now we analyze the terms $I_{1}$, \dots, $I_{4}$ in \eqref{eqn.i1234} in order to prove \eqref{eqn.trapcp}. We will focus on the assumptions and conclusions of item (i), item (ii) being treated along the same lines.

First we  observe  that
$I_{1}$ is exactly of the form \eqref{eqn.ri}, with $p<4$ and thus $m=3$. 
 Hence a direct application of Proposition \ref{prop:intg-ctrld-process} yields the almost 
 sure limit 
\begin{align*}
I_{1} \to \int_{0}^{T} y  \, d{\bf X} \qquad \text{as } |\cp|\to 0. 
\end{align*}
 
 Let us now analyze the term $I_{2}$ in \eqref{eqn.i1234}. To this aim, an elementary examination of matrix indices reveals that 
 \begin{align}\label{eqn.i2o}
I_{2} = y^{1}_{t_{k}} \cdot \lp \frac12 X^{1}_{t_{k}t_{k+1}} \otimes X^{1}_{t_{k}t_{k+1}} - X^{2}_{t_{k}t_{k+1}} \rp.   
\end{align}
Furthermore, since $X$ is a geometric rough path, notice that a consequence of \eqref{eq:multiplicativity} is that for all $(s,t)\in \cs_{2}([0,T])$ we have
\begin{align*}
\text{Sym} (X^{2}_{st}) = \frac12 X^{1}_{st}\otimes X^{1}_{st} . 
\end{align*}
Hence one can write \eqref{eqn.i2o} as 
\begin{align*}
I_{2} = - y^{1}_{t_{k}} \cdot \text{Antisym} (X^{2}_{t_{k}t_{k+1}}) 
=   \frac12 \sum_{i=1}^{d} y_{t_{k}}^{1,ij} \lp X_{t_{k}t_{k+1}}^{2, ji} - X_{t_{k}t_{k+1}}^{2, ij}  \rp_{ij} . 
\end{align*}
Thanks to our definition \eqref{F} of the increment $F$, this becomes 
\begin{align}\label{eqn.i2n}
I_{2} =   \frac12 \sum_{i=1}^{d}  y_{t_{k}}^{1,ij}  
\lp \delta F^{ji}_{t_{k}t_{k+1}}-\delta F^{ij}_{t_{k}t_{k+1}} \rp.
\end{align}
Hence owing to identity \eqref{eqn.i2n}, and since we have assumed that the increments of $y^{1}$ are dominated by the control $\omega$, Lemma  \ref{lem:weighted-F}   item (i) shows that 
\begin{align*}
\lim_{|\cp|\to 0} I_{2} = 0, \qquad \text{in probability.}
\end{align*}

We now handle the sum $I_{3}$ in \eqref{eqn.i1234}. To this aim, we   resort to Lemma \ref{lem:weighted-g} item (i) for the terms $y_{t_{k}}^{2}\cdot X^{3}_{t_{k}t_{k+1}}$ and to Remark \ref{remark.b} for the terms $y^{2}_{t_{k}} X^{2}_{t_{k}t_{k+1}} \cdot X^{1}_{t_{k}t_{k+1}} $. We end up with 
\begin{align*}
\lim_{|\cp|\to 0} I_{3} = 0, \qquad \text{in probability.}
\end{align*}

 In order to show   the convergence in \eqref{eqn.trapcp}, it remains to show that $I_{4}\to 0$ in probability. 
 Towards this aim, recall from Definition \ref{def:ctrld-process} that the increment $r^{0}_{t_{k}t_{k+1}}$ is dominated by $\omega (t_{k}, t_{k+1})^{3/p}$. Hence for a small $\ep>0$ we get  
 \begin{align}\label{eqn.i4bd}
|I_{4}| = \frac12 \sum_{0\leq t_{k}<T} \left| r^{0}_{t_{k}t_{k+1}} \cdot  X^{1}_{t_{k}t_{k+1}}\right| \leq \frac12 \sum_{0\leq t_{k}<T}  \omega (t_{k}, t_{k+1})^{3/p} \|X \|_{p\text{-var}, [t_{k}, t_{k+1}]}
\nonumber
\\
\leq \frac12 \max_{k}  \omega (t_{k}, t_{k+1})^{\ep} \cdot  \sum_{0\leq t_{k}<T}  \omega (t_{k}, t_{k+1})^{3/p-\ep} \|X \|_{p\text{-var}, [t_{k}, t_{k+1}]}. 
\end{align}
Now set 
\begin{align*}
\tilde{\omega}(s,t) = \omega (s,t)^{3/p-\ep} \cdot\|X \|_{p\text{-var}, [s,t]}
\qquad (s,t)\in \cs_{2}.
\end{align*}
Using the same argument as for \eqref{eqn.drmbd}, 
 it is easy to see that $\tilde{\omega}$ is a control. Therefore, by the super-additivity of $\tilde{\omega}$ we get
\begin{align*}
|I_{4}|  \leq \frac12 \max_{k}  \omega (t_{k}, t_{k+1})^{\ep} \cdot   \omega (0,T)^{3/p-\ep}\cdot \|X \|_{p\text{-var}, [0,T]}. 
\end{align*}
Since $\max_{k}\omega(t_{k}, t_{k+1})\to 0$ as $|\cp|\to0$ it follows that $I_{4}\to 0$ almost surely.  
This completes the proof of \eqref{eqn.trapcp}. 
 Moreover, recall that claim (ii) in our statement is obtained easily by adapting slightly the considerations above, similarly to what we have done in Lemma \ref{lem:weighted-F}. This completes the proof of our theorem.
  \end{proof}
 
 As mentioned in Theorem \ref{thm:cvgce-trapezoid-intro}, typical examples of controlled processes are given by solutions of rough differential equations and processes of the form $y=f(X)$. Hence one can apply Theorem \ref{theorem:Trapezoid-Rule} in order to get a trapezoid rule \eqref{eq:proof-what-to- prove} for $f(X)$. However, we would also like to consider Riemann sums which are closer to the ones handled in \cite{BNN, HN}. This is why we wish to consider sums fo the form:  
   \begin{align}\label{eqn.m}
\text{m-}\mathcal{J}_{0}^{T} (f(X), X) :=   \sum_{k=0}^{n-1} f\left( \frac{X_{t_{k}}+X_{t_{k+1}}}{2} \right)   \delta X_{t_{k}t_{k+1}}. 
\end{align}
We now state  a corollary of Theorem \ref{theorem:Trapezoid-Rule} giving the convergence of m-$\cj_{0}^{T} (f(X), X)$ above. 
 	\begin{cor}\label{cor.mid}
	Let  $X$ be as in Theorem \ref{theorem:Trapezoid-Rule}. Consider   function  $f\in C^{3}_{b}(\RR^{d} )$    
	and the midpoint rule m-$ \cj_{0}^{T} (f(X), X)$ defined by \eqref{eqn.m}.
	 Then we have
	\begin{align}
	\label{eqn.mid}
\emph{m-}\mathcal{J}_{0}^{T} (f(X),X)  \to \int_{0}^{T} f(X_{s})   d{\bf X}_{s}
\end{align}
  as the mesh size $|\mathcal{P}|\to 0$. As in Theorem \ref{theorem:Trapezoid-Rule}, the convergence holds in probability if      
$p$-variation regularity  is considered,  and   almost surely if  H\"older continuity  is assumed. 
	\end{cor}
	\begin{proof}
	We first recall that for $a, b\in \R^{d}$ we have   the following  mean value   identity 
\begin{align}\label{eqn.mv}
 \frac{f(a)+f(b)}{2}- f\Big(\frac{a+b}{2}\Big)  = \frac12  \partial^{2}f (c) \Big(\frac{b-a}{2}  \otimes  \frac{b-a}{2}\Big)  
\end{align}
  where $c\in \R^{3} $ satisfies  $c=a+\theta  (b-a)$ for some $\theta\in [0,1]$. Let us take   the difference  of \eqref{eq:proof-what-to- prove} and \eqref{eqn.m}  
and then apply the mean value identity   \eqref{eqn.mv} with $a=X_{t_{k}}$, $b=X_{t_{k+1}}$.  Then we   obtain   
  \begin{align}\label{eqn.tr-m}
   \text{tr-}\mathcal{J}_{0}^{T} (f(X), X) - \text{m-}\mathcal{J}_{0}^{T} (f(X), X) 
  =
\frac18    \sum_{k=0}^{n-1}  \partial^{2}f(c)
\lp \delta X_{t_{k}t_{k+1}} \otimes  \delta X_{t_{k}t_{k+1}}
\rp
   \cdot  \delta X_{t_{k}t_{k+1}}  
 \,,
\end{align}
 with $c = X_{t_{k}}+\theta \delta X_{t_{k}t_{k+1}}$. 
 In order to prove \eqref{eqn.mid}, it suffices to show that the right-hand side of \eqref{eqn.tr-m} converges to zero.  To this aim, we observe that  
 \begin{align}\label{eqn.fpp}
\partial^{2}f(c) = \partial^{2}f(X_{t_{k}}) + \theta \partial^{3}f(d) \delta X_{t_{k}t_{k+1}},  
\end{align}
where $d = X_{t_{k}}+\lambda \delta X_{t_{k}t_{k+1}} $ for some $\lambda \in [0,1]$. 
Substituting \eqref{eqn.fpp} into the right-hand side of \eqref{eqn.tr-m} we obtain two terms. It is then easy to see that one of the two terms is in the form of $\sum_{0\leq t_{k}<T} y_{t_{k}} X^{3}_{t_{k}t_{k+1}}$ with $y =f(X)$. It then follows from Lemma  \ref{lem:weighted-g} that it converges to zero. The other term can be treated in a similar way as for $I_{4}$ in \eqref{eqn.i4bd}, which  completes the proof. 
     \end{proof}

\subsection{Applications}\label{sec:applications}
In this section we will briefly list some important examples of Gaussian processes satisfying our standing Hypothesis \ref{hyp.x}. Notice that in the current paper we only request $V_{\rho}(R)<\infty$ with $\rho\in [1,2)$, which is a weaker condition than in \cite{JM}, and certainly weaker than in \cite{GOT}. Hence all the examples listed in those two references also apply to our context. We just highlight some of them below. 

\begin{enumerate}[wide, labelwidth=!, labelindent=0pt, label=\textbf{(\roman*)}]
\setlength\itemsep{.1in}

\item
The most obvious example is given by a fractional Brownian motion (fBm) $B^{H}$ for which the covariance function $R$ in \eqref{eq:def-covariance-X} is given by 
\begin{align*}
R^{H}(s,t) = \frac12 \lp t^{2H}+s^{2H}-|t-s|^{2H} \rp. 
\end{align*}
Then $R^{H}$ satisfies Hypothesis \ref{hyp.x} whenever $H\in (\frac14, 1)$, with $\omega_{R}([s,t]^{2})=|t-s|$. One can also verify that \eqref{eqn.Rholder} holds. 

\item
If one considers a process $X$ given as $X=B^{H_{1}}+B^{H_{2}}$ with $H_{1}, H_{2}\in (\frac14, 1)$ and two independent $\R^{d}$-valued fBms $B^{H_{1}}$ and $B^{H_{2}}$, then one can also apply our main Theorem \ref{thm:cvgce-trapezoid-intro} to $X$, with $R(s,t) = R^{H_{1}}(s,t)+R^{H_{2}}(s,t)$. 

\item
The bifractional Brownian motion, introduced in \cite{HV} is a centered Gaussian process whose covariance $R^{H,K}$ is given by  
\begin{align*}
R^{H,K}(s,t)   =\frac{1}{2^K}\left((t^{2H}+s^{2H})^K-|t-s|^{2HK}\right). 
\end{align*}
This process generalizes fBm (obtained for $K=1$), and fulfills our Hypothesis \ref{hyp.x}  whenever $HK\in (1/4,1)$.

\item
We refer to \cite{JM} for a thorough exploration of random Fourier series, some of which yield  a control such that $\omega_{R}([s,t]^{2}) \neq |t-s|$, but still satisfying Hypothesis \ref{hyp.x}. 
	 
\end{enumerate}


\begin{thebibliography}{9}

	\bibitem{BFG}
	  Bayer, C.;    Friz, P. and    Gatheral, J. (2016). Pricing under rough volatility, 
	\textit{Quant.  Finance} \textbf{16} no. 6, 887-904. 
	
	\bibitem{BNN} Binotto, G.; Nourdin, I. and Nualart, D. (2018). Weak symmetric integrals with respect to the fractional Brownian motion. \textit{Ann. Probab.} \textbf{46} no. 4, 2243-2267. 
	
	\bibitem{Alg-Renorm}
	 Bruned, Y.;    Hairer, M. and   Zambotti, L. (2019). 
	Algebraic renormalisation of regularity structures \textit{Invent. Math.} \textbf{215} no. 3, 1039-1156. 
	
	\bibitem{Ch}
	  Chen, K. (1954). 
Iterated integrals and exponential homomorphisms.	\textit{Proc. London Math. Soc.} \textbf{3} no. 4, 502-512. 

	
	\bibitem{RF}
	  Diehl,  J.;  Oberhauser, H. and   Riedel, S.
	(2015). A L\'evy area between Brownian motion and rough paths with applications to robust nonlinear filtering and rough partial differential equations. \textit{Stochastic Process. Appl.} \textbf{125} no. 1, 161-181. 
	
	\bibitem{JM}
	  Friz, P.;    Gess, B.;  Archil Gulisashvili, S. and    Riedel, S. (2016). 
	The Jain-Monrad criterion for rough paths and applications to random Fourier series and non-Markovian H\"ormander theory. \textit{Ann. Probab.} \textbf{44} no. 1, 684-738. 
	
	\bibitem{HairerBook}
	  Friz, P. and   Hairer, M. (2014). 
	\textit{A Course on Rough Paths.} With an introduction to regularity structures. Universitext. Springer, Cham.  
	
	\bibitem{FrizBook}
	  Friz, P. and  Victoir, N. 
(2010).
 \textit{Multidimensional stochastic processes as rough paths. Theory and applications.} Cambridge Studies in Advanced Mathematics, 120. Cambridge University Press, Cambridge.  


	\bibitem{FV}
	Friz, P. and Victoir, N. (2011).  A note on higher dimensional $p$-variation. \textit{Electron. J. Probab.} {\bf 16}  no. 68, 1880-1899.
	
	
	
	\bibitem{GOT}
	  Gess, B.;    Ouyang, C. and   Tindel, S. (2017). 
Density bounds for solutions to differential equations driven by Gaussian rough paths	\textit{preprint.} 
	
	\bibitem{GNRV} Gradinaru, M.; Nourdin, I.; Russo, F. and Vallois, P. (2005). $m$-order integrals and generalized It\^o's formula: the case of a fractional Brownian motion with any Hurst index.
\textit{Ann. Inst. H. Poincar\'e Probab. Statist.} \textbf{41} no. 4, 781-806.



\bibitem{ChineseChar}
	  Graham, B. (2013). 
	Sparse arrays of signatures for online character recognition. \textit{preprint.}
	
	\bibitem{HN} 
	Harnett, D. and Nualart, D. (2012). Weak convergence of the Stratonovich integral with respect to a class of Gaussian processes. \textit{Stochastic Process. Appl.} \textbf{122}, no. 10, 3460-3505. 

\bibitem{HV} Houdr\'e, C. and Villa, J. (2003). An example of infinite dimensional quasi-helix. \textit{Contemporary Mathematics, Amer. Math. Soc.}, \textbf{336},   195-201.	
	
\bibitem{HLN} Hu, Y.; Liu, Y. and Nualart, D. (2019). Crank-Nicolson scheme for stochastic differential equations driven by fractional Brownian motions.   \textit{Ann. Appl. Probab.} To appear. 

\bibitem{HLN2} Hu, Y.; Liu, Y. and Nualart, D. (2016). Rate of  convergence and asymptotic error distribution of Euler approximation schemes for fractional diffusions.    \textit{Ann. Appl. Probab.} \textbf{26}, no. 2, 1147-1207. 

\bibitem{HLN3} Hu, Y.; Liu, Y. and Nualart, D. (2016). Taylor schemes for rough differential equations and fractional diffusions.    \textit{Discrete Contin. Dyn. Syst. Ser. B} \textbf{21}, no. 9, 3115-3162.



	\bibitem{LD}
	  Ledoux, M.;    Qian, Z. and    Zhang, T. (2002).  Large deviations and support theorem for diffusion processes via rough paths.
	\textit{Stochastic Process. Appl.} \textbf{102} no. 2, 265-283. 
	
	\bibitem{OneD}
	  Liu, Y. and   Tindel, S. (2019). 
	Discrete rough paths and limit theorems. \textit{Annales de l'Institut Henri Poincar\'e Probabilit\'es et Statistiques.} To appear.  	
	
	\bibitem{Euler}
	  Liu, Y. and   Tindel, S. (2019). 
	 First-Order Euler Scheme for SDEs Driven By Fractional Brownian Motions: The Rough Case. \textit{Ann. Appl. Probab.} \textbf{29} no. 2, 758-826. 
	
	
	\bibitem{Lyons}
	  Lyons, T. (1998). 
	Differential equations driven by rough signals. \textit{Rev. Mat. Iberoamericana} \textbf{14} no. 2, 215-310. 
	
	\bibitem{NN}
	Neuenkirch, A. and Nourdin, I. 
	(2007).  Exact rate of convergence of some approximation schemes associated to SDEs driven by a fractional Brownian motion. 
	\textit{J. Theoret. Probab.} \textbf{20}  no. 4, 871-899.
	
	\bibitem{NuaBook}
	  Nualart, D. (2006). 
	\textit{The Malliavin Calculus and Related Topics.} Second edition. 
	
	\bibitem{NuaTin}
	  Nualart, D. and    Tindel, S. 
	(2011). A construction of the rough path above fractional Brownian motion using Volterra's representation. \textit{Ann. Probab.} \textbf{39} no. 3, 1061-1096. 
	
	\bibitem{fBm}
	  Nourdin, I. (2012). 
	\textit{Selected Aspects of Fractional Brownian Motion.} Bocconi \& Springer Series, 4. Springer, Milan; Bocconi University Press, Milan. 
	 
	\bibitem{NR} Nourdin, I. and R\'eveillac, A. (2008). Asymptotic behavior of weighted quadratic variations of fractional Brownian motion: The critical case $H=1/4$. \textit{Ann. Probab.}  \textbf{37} no. 6,   2200-2230.
	
	\bibitem{NRS} Nourdin, I.; R\'eveillac, A. and Swanson, J. (2010). The weak Stratonovich integral with respect to fractional Brownian motion with Hurst parameter. \textit{Electron. J. Probab.} \textbf{15} no. 70, 2117-2162. 
	
\bibitem{RV} Russo, F. and Vallois, P. (1993). Forward, backward and symmetric stochastic integration.
\textit{Probab. Theory Related Fields} \textbf{97} no. 3, 403-421.

\bibitem{To}
Towghi, N. (2002). 
Multidimensional extension of L. C. Young's inequality,.
\emph{J. Inequal. Pure Appl. Math.} {\bf 3} no. 2, Article 22, 13 pp. (electronic).


	
	\bibitem{Young}
	  Young, L. 
	(1936). An inequality of the H\"older type, connected with Stieltjes integration. \textit{Acta Math.} \textbf{67} no. 1, 251-282. 



	
	
\end{thebibliography}
\end{document}